\documentclass[12pt]{amsart}
\usepackage{amssymb}
\usepackage{amsmath,amscd}
\usepackage{amsfonts,latexsym,verbatim,amscd,mathrsfs,color,array}
\usepackage[all,cmtip]{xy}
\usepackage{mathrsfs}
\usepackage{multirow}
\usepackage{graphicx}
\usepackage{amsfonts, amsthm}
\usepackage{bbm}
\usepackage[hmargin=1in,vmargin=1in]{geometry}
\usepackage{mathdots}
\usepackage{stmaryrd}
\usepackage{MnSymbol}
\usepackage{bbm}
\usepackage[hidelinks]{hyperref}
\usepackage{enumitem}
\usepackage[all]{xy}
\usepackage{tikz-cd}

\allowdisplaybreaks

\def\XXint#1#2#3{{\setbox0=\hbox{$#1{#2#3}{\int}$ }
		\vcenter{\hbox{$#2#3$ }}\kern-.6\wd0}}

\newcommand{\transint}{\cap\kern-0.63em|\kern0.7em}

\DeclareMathSymbol{\intprod}{\mathbin}{MnSyC}{'270}

\newcommand{\p}{{ \partial}}

\renewcommand{\1}{{\mathbf 1}}

\newcommand{\R}{{\mathbb R}}

\renewcommand{\p}{{\partial}}

\newtheorem{thm}{Theorem}[section]
\newtheorem{lemma}[thm]{Lemma}
\newtheorem*{lemma*}{Lemma}
\newtheorem{prop}[thm]{Proposition}

\newtheorem{cor}[thm]{Corollary}

\newtheorem*{conj*}{Conjecture}

   \newtheoremstyle{others}
     {3pt}
     {2pt}
     {}
     {}
     {\bf}
     {.}
     {.5em}
     {}

\theoremstyle{others}

\newtheorem*{rmk*}{Remark}
\newtheorem{defn}[thm]{Definition}

\newtheorem*{question*}{Question}

\numberwithin{equation}{section}

\usepackage{mathtools} 
\usepackage{faktor}
\usepackage{dynkin-diagrams}
\usepackage{mathrsfs}
\usepackage{hyperref}
\usepackage{tikz-cd}
\usepackage{bbm}

\usepackage{romanbar}

\newtheorem{definition}[thm]{Definition}

\begin{document}

\title[Semilinear heat inequality with 
Morrey initial data]{The semilinear heat inequality with Morrey initial data on 
Riemannian manifolds}
\author{Anuk Dayaprema}
\begin{abstract}
    The goal of this paper is to obtain estimates for nonnegative solutions of the differential inequality
\begin{equation}
        \left(\frac{\partial}{\partial t} - \Delta\right) u \leq A u^p + B u \nonumber
    \end{equation}
    with small initial data in borderline Morrey norms over a 
    Riemannian manifold with bounded geometry. 
    We obtain \(L^\infty\) estimates 
    assuming $$\|u(\cdot,0)\|_{M^{q, \frac{2q}{p-1}}} + \sup_{0 \leq t < T} \|u(\cdot, t) \|_{L^s} < \delta,$$
    where \(1 < q \leq q_c := \frac{n(p-1)}{2}\) and \(1 \leq s \leq q_c\). Assuming also a bound on $\|u(\cdot, 0)\|_{M^{q', \lambda'}}$, where $\frac{\lambda'}{2q'} < \frac{1}{p-1},$ 
    we get an improved estimate near the initial time. These results have applications to geometric flows in higher dimensions.
\end{abstract}
\maketitle

\vspace{-.5cm}
\section{Introduction} 
\subsection{Background}\label{subsection:background}

\thispagestyle{empty}

Semilinear heat equations have been studied extensively on domains in Euclidean space; see the monograph by Quittner and Souplet \cite[Chapter  
 II]{superlinearbook}.
 
 In 1980, Weissler \cite{weisslersemilinearLpwellposed} showed that for \(p > 1\) and a bounded domain \(\Omega \subset \R^n\), the Cauchy-Dirichlet problem
 \begin{align}\label{diffeq}
     \left( \frac{\p}{\partial t} - \Delta \right) u = |u|^{p-1}u, \quad u(\cdot, 0) = u_0, \quad u\big|_{\partial \Omega} = 0,
\end{align}
is well-posed for initial data \(u_0 \in L^q\), with \(q\) strictly bigger than the critical value
 \begin{align*}
     q_c := \frac{n(p-1)}{2}.
 \end{align*} For initial data in the scale-invariant/critical space \(L^{q_c}(\Omega)\), with \(q_c > 1\), Weissler also showed that one can classically solve (\ref{diffeq}), although the interval of definition is no longer uniformly bounded away from zero on bounded subsets of \(L^{q_c}(\Omega)\). In fact, it follows from work of Kaplan \cite{kaplan} and a rescaling argument \cite[Remark 15.4 (i) and Theorem 17.1]{superlinearbook} that, for bounded domains, the \(L^{q_c}\) norm does not control the maximal time of existence. \par
 However, in 1999, Souplet \cite{soupletstability} showed that if \(p > 1 + \frac{2}{n}\) 
 and the \(L^{q_c}\)-norm of the initial data is sufficiently small, then the solution exists for all time and converges to the zero function at a polynomial rate. \par
 
 The problem (\ref{diffeq}) with small, borderline initial data has been further studied in the setting of Morrey spaces. For given \(n\) and \(p\), there is a one-parameter family of scale-invariant/critical Morrey spaces \(M^{q, \frac{2q}{p-1}}(\Omega)\), with \(1 \leq q \leq q_c\), all of which include \(L^{q_c}(\Omega)\).  
  In 2015, Blatt and Struwe \cite{blattstruweanalyticframework, blattstruweboundary, laneemdenMorrey} showed that on bounded 
  domains in \(\R^n\), with \(n \geq 3\) and \(p > \frac{n+2}{n-2}\), (\ref{diffeq}) has a global-in-time classical solution if \(u_0\) is small in the \(M^{2, \frac{4}{p-1}}\) norm. Furthermore, it follows from their work that the solution is bounded for all positive times and decays polynomially at the same rate as in the \(L^{q_c}\) case. Additionally, Blatt and Struwe derived a relaxed local well-posedness result, similar to Weissler's for \(L^{q_c}(\Omega)\), in parabolic Morrey spaces for initial data in a proper subspace of \(M^{2, \frac{4}{p-1}}(\Omega)\). \par
 In 2016, Souplet \cite{semilinearmorrey} generalized the work of Blatt and Struwe to a wider range of Morrey exponents. In particular, Souplet showed using semigroup methods that initial data which are small in the \(M^{q, \frac{2q}{p-1}}\) norm, with \(1 < q \leq q_c\), yield a global classical solution to (\ref{diffeq}). Moreover, the \(L^\infty\) norm of the solution decays polynomially at the same rate as in the \(L^{q_c}\) case. In addition, Souplet's approach yields that for \(q\) in the aforementioned range, the \(M^{q, \frac{2q}{p-1}}\) norm of the solution remains small up to a multiplicative constant. \par
Many geometric flows
involve one or more nonnegative energy density functions subject to differential inequalities of the form
\begin{align}\label{diffineq}
    \left( \frac{\p}{\p t} - \Delta \right) u \leq Au^p + Bu,
\end{align} with nonnegative constants \(A\) and \(B\). These flows often exist smoothly as long as the energy density remains bounded; hence, 
it is desirable to have long-time $L^\infty$ estimates for solutions of (\ref{diffineq}) on Riemannian manifolds under the weakest possible assumptions. The main goal of this paper is to adapt several of the aforementioned results for the equation (\ref{diffeq}) to solutions of the inequality (\ref{diffineq}) on a 
Riemannian manifold with bounded geometry.

\subsection{Statement of results}\label{subsection:statementofresults}

In the sequel, \(M\) is a connected, complete 
Riemannian manifold of dimension \(n \geq 3\) with bounded geometry and \(A,B\) are nonnegative constants. We use \(M^{q, \lambda}\) to denote Morrey spaces of functions and \(H_t(f)\) to denote the integral of a function \(f\) against the heat kernel on 
\(M\). 
See \textsection \ref{section:preliminaries} for definitions. \par 
We first give the hypotheses on the solutions to (\ref{diffineq}) that we herein consider.
\begin{definition}\label{defn:weaksolution}
        Let \(p, q, r\), and \(T\) be constants such that
        \begin{align*}
            p > 1 + \frac{2}{n}, \quad 1 < q \leq q_c := \frac{n(p-1)}{2}, \quad 1 \leq \frac{r}{p} < q < r, \quad T > 0.
        \end{align*}Set
        \begin{align*}
            \lambda := \frac{2q}{p-1}.
        \end{align*} We say \(u\) is a solution to (\ref{diffineq}) on \([0, T)\) if:
        \begin{enumerate}
            \item \(u \in C_{\text{loc}}^0([0, T), M^{q, \lambda}) \cap C_{\text{loc}}^0((0, T), L^\infty)\). 
            \item \(u(t)\) satisfies 
        \begin{align*}
            \lim_{t \searrow 0} \left(\|u(t) - u(0)\|_{M^{q, \lambda}} + t^{\frac{1}{p-1}} \|u(t)\|_\infty + t^{\frac{\lambda}{2}\left(\frac{1}{q} - \frac{1}{r}\right)}\|u(t)\|_{M^{r, \lambda}}\right)  = 0.
        \end{align*}
        \item For each \(t \in [0, T)\), we have \(u(t) \geq 0\) a.e., and for \(0 \leq t' < t < T\) we have a.e.
        \begin{align}\label{weaksolution:subsolution}
            u(t) \leq H_{t-t'}(u(t')) + \int_{t'}^t H_{t-s}(Au^p(s) + Bu(s)) \, ds.
        \end{align}
        \end{enumerate}
    \end{definition}
\begin{thm} \label{thm:Morreysupbound}
There exists \(\delta_0 > 0\), depending on \(n, p, q, r, A, B\), and the geometry of \(M\), such that the following holds. \par 
Let \(\delta \in (0, \delta_0]\) and \(s \in [1, q_c]\). Suppose \(u\) is a solution of (\ref{diffineq}) on \([0, T)\), satisfying
		\begin{align}\label{Morreysupbound:criticalinitialhypothesis}
		\|u(0)\|_{M^{q, \lambda}} + 
        \sup_{(x, t) \in M \times [0, T)} \|u(t)\|_{L^s(B_1(x))} < \delta.
		\end{align} 
  Then, 
		\begin{align}\label{Morreysupbound:criticalestimate}
		 \|u(t)\|_{L^\infty} \leq C_{(\ref{Morreysupbound:criticalestimate})}\delta \left( t^{-\frac{1}{p-1}} + 1 \right), \ \ \ 0 < t < T. 
		\end{align}
  Here, \(C_{(\ref{Morreysupbound:criticalestimate})}\) is defined by (\ref{Morreysupbound:longconst}) and depends on \(n, B\), and the geometry of \(M\). \par
  Let $K > 0,$ and let $q' \in [1, \infty]$ and \(\lambda' \in [0, n]\) satisfy
    $$\alpha := \frac{\lambda'}{2q'} < \frac{1}{p-1}.$$
  Suppose that in addition to (\ref{Morreysupbound:criticalinitialhypothesis}) we have
  \begin{align}\label{Morreysupbound:improvedauxi}
       u \in C_{\text{loc}}^0([0, T), M^{q', \lambda'}), \quad \|u(0)\|_{M^{q', \lambda'}} \leq K, \quad  \lim_{t \searrow 0} t^{\alpha}\|u(t)\|_\infty = 0.
  \end{align} Then,
  \begin{align}\label{Morreysupbound:improvedestimate}
		\|u(t)\|_{L^\infty} \leq C_{(\ref{Morreysupbound:improvedestimate})} t^{-\alpha} \wedge C_{(\ref{Morreysupbound:criticalestimate})}\delta  t^{-\frac{1}{p-1}} + C_{(\ref{Morreysupbound:criticalestimate})}\delta , \ \ \ 0 < t < T.
    \end{align}
    Here, \(C_{(\ref{Morreysupbound:improvedestimate})}\) is defined by (\ref{Morreysupbound:improvedconst}) and depends on \(K, n, p, \alpha, A, B,\) and the geometry of \(M\).
\end{thm}

\subsection{Applications} 
We first give an application of the preceding result to harmonic map flow. Recall that the energy density \(|dw|\) of a smooth solution \(w(t) : M \to N\) of harmonic map flow between closed Riemannian manifolds \(M, N\) satisfies (\ref{diffineq}) with \(p = 3\).  By work of Wang \cite{hmfLngradient, hmfbmo}, we have that given a Sobolev map \(w_0 \in W^{1, n}(M, N)\), where \(n = \text{dim}(M)\), there exists a 
short-time regular solution of harmonic map flow. 
We show that the energy density of such a solution satisfies the conditions of Definition \ref{defn:weaksolution}, 
allowing us to 
prove the following 
version of Struwe's classical small-data regularity result for harmonic map flow \cite[Theorem 7.1]{struwehigherdimhmf}. 

\begin{thm}\label{thm:regularballtheorem}
Let \(M\) and \(N\) be closed Riemannian manifolds, and let \(w_0 \in W^{1, n}(M, N)\), where \(n = \dim(M)\). There exists \(\delta_1 > 0\), depending on the geometries of \(M\) and \(N\), such that if \(\delta \in (0, \delta_1]\) and \(\|dw_0\|_{M^{2, 2}} < \delta\), 
then Wang's solution of harmonic map flow with $w(0) = w_0$ exists for all time and converges exponentially in \(C^k\) for all \(k\) to a constant map as \(t \to \infty\).
\end{thm}

In a companion paper \cite{gappaper}, we develop analogous results for Yang-Mills flow in higher dimensions: In particular, \cite[Theorem 1.4]{gappaper} is directly analogous to Theorem \ref{thm:regularballtheorem}. 
We note that Theorem \ref{thm:regularballtheorem} in this paper and \cite[Theorem 1.4]{gappaper} also follow from a monotonicity formula and epsilon-regularity result for the energy densities of the respective flows. However, in the context of Yang-Mills flow on a special-holonomy manifold, such a monotonicity formula does not hold for the individual components of the curvature and the estimates of the present paper are essential; see \cite[Theorem 1.3]{gappaper}. 

Last, we use Theorem \ref{thm:regularballtheorem} to 
give a version of a theorem of White \cite[Cor. 2 and 3 on pg. 128]{whitejdg} concerning the infimum of the functional \(\int_M |dw|^n\) on homotopy classes of maps between closed manifolds. 
\begin{cor}\label{gapcor}
    Let \(M\) and \(N\) be closed Riemannian manifolds and $\gamma \subset C(M,N)$ a homotopy class of maps between $M$ and $N.$ 
    Then, $\gamma$ is trivial if and only if
    \begin{align}\label{Morreyinfimum}
        \inf\{ \|dw\|_{M^{2, 2}(M)} \mid w \in \gamma \cap W^{1, n}(M, N) \} = 0. 
    \end{align}
\end{cor}
For the gauge-theoretic analogue of Corollary \ref{gapcor}, see \cite[Corollary 1.5]{gappaper}.
\section{Preliminaries}\label{section:preliminaries}
Let \(M\) be a 
Riemannian manifold of dimension \(n\) with Riemannian metric \(g\). We denote the covariant derivative induced by the Levi-Civita connection of \(g\) by \(\nabla\). Our convention for the Laplace-Beltrami operator is given by the following action on functions:
\begin{align*}
    \Delta f = g^{ij}\nabla_i \nabla_j f = \frac{1}{\sqrt{\det g}} \partial_i \left(\sqrt{\det g}g^{ij}\partial_j f\right).
\end{align*} We let \(\mu_g\) denote the Riemannian measure induced by \(g\).
\par 
We next define Morrey spaces of functions.
Suppose \(M\) has bounded geometry, i.e., the sectional curvatures are uniformly bounded and the injectivity radius is uniformly bounded away from zero.
It follows from the volume comparison theorems of G{\"u}nther and Bishop-Gromov 
that there exists \(V \geq 1\) such that 
for all \(R \in (0, 1)\) and \(x \in M\) we have
 \begin{align}\label{preliminaries:volume}
     V^{-1}R^n \leq 
     \mu_g(B_R(x)) \leq VR^n.
 \end{align} 
        \begin{defn}\label{defn:Morreydefinition}
		 Given \(q \in [1, \infty]\) and \(\lambda \in [0, n]\), we define the Morrey norm of a locally integrable function \(f\) to be 
    \begin{align*}
        \|f\|_{q, \lambda} := 
        \begin{cases}
            \sup_{x \in M, 0 < R \leq 1} \left(R^{\lambda-n} \int_{B_R(x)} |f|^q \, \mu_g\right)^{\frac{1}{q}} & q < \infty \\
            \|f\|_\infty & q = \infty.
        \end{cases}
    \end{align*} We define the Morrey space \(M^{q, \lambda}\) to be \(L^\infty\) if \(q = \infty\), and otherwise
		\begin{align}
		M^{q, \lambda} = \left\{f \in L_{\text{loc}}^1 \mid \|f\|_{q, \lambda} < \infty\right\}. \nonumber
		\end{align}
	\end{defn} 
	
It follows from H{\"o}lder's inequality that if \(q_1 \leq q_2\) and \(\lambda_2 \leq \frac{q_2}{q_1}\lambda_1\), then 
    \begin{align}\label{preliminaries:Morreyembedding}
        M^{q_2, \lambda_2} \subset M^{q_1, \lambda_1},
    \end{align} and the inclusion has operator norm bounded by a constant depending only on \(V\). 

 \par 
  Finally, we mention the heat kernel on manifolds. For a manifold with bounded geometry, there is a positive function
  \begin{align*}
      H \in C^\infty(M \times M \times (0, \infty))
  \end{align*} which is the unique minimal fundamental solution of the heat equation on \(M\) \cite[Theorem 4 in Chapter VIII.2]{chaveleigenvaluesbook}. We call \(H\) the heat kernel on \(M\), and we denote
    \begin{align*}
        H_t(f)(x) := \int_M H(x, y, t) f(y) \, \mu_g(y).
    \end{align*} 
\section{Initial data in Morrey spaces}\label{section:Morrey}

        We now adapt \cite[Prop. 49.17]{superlinearbook}, which gives estimates in Morrey spaces for solutions to the heat equation for Euclidean domains with Dirichlet boundary conditions, to the Riemannian setting.
        The following lemma lists properties of the heat kernel on 
        manifolds with bounded geometry that we will use.
        \begin{lemma}\label{lemma:heatkernel}
        Suppose \(M\) has bounded geometry. Then, there exist constants \(C_u, C_l > 0\), depending only on the geometry of \(M\), such that, letting  \(d(x, y)\) denote the Riemannian distance between \(x, y \in M\), the heat kernel on \(M\) satisfies
		\begin{align}\label{heatkernel:upperbound}
		  H(x, y, t) \leq C_ut^{-\frac{n}{2}}\exp\left(-\frac{d(x, y)^2}{C_ut}\right), \ \ \ 0 < t \leq 1,
		\end{align} and        \begin{align}\label{heatkernel:lowerbound}
            t^{-\frac{n}{2}}\exp\left(-\frac{d(x, y)^2}{2t}\right) \leq C_lH(x, y, t), \ \ \ 0 < t \leq 18C_u.
        \end{align} Furthermore,	\begin{align}\label{heatkernel:stochasticcompleteness}
		\int_M H(x, y, t) \, \mu_g(y) = 1, \ \ \ \forall \ x \in M, \  t > 0. 
		\end{align} 
    \end{lemma}
\begin{proof}
    The bound (\ref{heatkernel:upperbound}) follows from \cite{chengliyauheatkernel}, and the bound (\ref{heatkernel:lowerbound})  follows from \cite{cheegeryau}; see \cite[\textsection A.4]{semilinearivpriem} for an exposition. The property (\ref{heatkernel:stochasticcompleteness}) follows from \cite{yauheatkernel}.
\end{proof}

\begin{prop}\label{prop:Morreyheatestimates}
	Let \(s_1, s_2,\) and \(\lambda\) be constants such that \(1 \leq s_1  \leq  s_2 \leq \infty\) and \(0 \leq \lambda \leq n\). Let \(M\) be a Riemannian manifold of dimension \(n\) with bounded geometry. Then, there exists a constant \(C_{(\ref{Morreyheatestimates:estimate})} \geq 1\), depending only on the geometry of \(M\), such that for all \(f \in M^{s_1, \lambda}\),
	\begin{align}\label{Morreyheatestimates:estimate}
	\|H_t(f)\|_{s_2, \lambda} \leq C_{(\ref{Morreyheatestimates:estimate})}t^{-\frac{\lambda}{2}(\frac{1}{s_1} - \frac{1}{s_2})}\|f\|_{s_1, \lambda}, \ \ 0 < t \leq 1.
	\end{align}
	\end{prop}
	\begin{proof} 
                Since \(M\) has bounded geometry, for estimating the Morrey norm, 
                we may restrict our attention to balls of radius \(R \in (0, R_0]\) for some fixed \(R_0\) depending on \(M\).
                We shall assume $R_0$ is sufficiently small that the following hold for some dimensional constant \(C_n \geq 2\).
                \begin{enumerate}
                    \item \label{Morreyheatestimates:rhoinfinity} Let \(\kappa > 0\) be such that the absolute value of any sectional curvature of \(M\) is bounded above by \(\kappa\). We then take \(R_0 \leq \rho_\infty := 4^{-1}\min\left\{1, \text{inj}(M), \kappa^{-1}\pi\right\} > 0\), as in  \cite[(2.2)]{semilinearivpriem}.
                    \item \label{Morreyheatestimates:uniformeuclidean} The metric in geodesic coordinates on any ball of radius \(6R_0\) satisfies
                    \begin{align*}
                        \frac{1}{2}g_{\text{Euc}} \leq g \leq 2g_{\text{Euc}}.
                    \end{align*} 
                    \item \label{Morreyheatestimates:uniformtransition} For any $x, y$ with $d(x,y) \leq 5R_0$, the Jacobian determinant of the  transition map between geodesic coordinates on the balls of radius \(6R_0\) centered at \(x, y\) is bounded above by \(C_n\).
                    \item \label{Morreyheatestimates:uniformcover} For each \(R \in (0, R_0]\), there exists a cover of $M$ by balls $B_{R}(p_i)$ such that for each \(j\), no more than \(C_n\) of the balls in the cover intersect \(B_R(p_{j})\). We also assume that any ball of radius \(2R\) can be covered by \(C_n\) balls of radius \(R\) (not necessarily in the cover).
                \end{enumerate}

We break the proof of (\ref{Morreyheatestimates:estimate}) into the following four cases:
\begin{enumerate}
    \item[(i)] \(s_1 = s_2 = \infty\);
    \item[(ii)] \(s_1 = s_2 \in [1, \infty)\);
    \item[(iii)] \(s_1 \in [1, \infty), s_2 = \infty\);
    \item[(iv)] \(1 \leq s_1 < s_2 < \infty.\)
\end{enumerate} \par
\emph{Case i: \(s_1 = s_2 = \infty\).}
This follows, with \(C_{(\ref{Morreyheatestimates:estimate})} = 1\), from (\ref{heatkernel:stochasticcompleteness}) since \(M^{\infty, \lambda} = L^\infty\). \par 
\emph{Case ii: \(s_1 = s_2 \in [1, \infty)\).}
   Let \(x_0 \in M\) and \(R \in (0, R_0]\). We first show that 
   \begin{align}\label{Morreyheatestimates:case2step1}
       \int_{B_R(x_0)} \int_M H(x, y, t) |f(y)|^{s_1} \, \mu_g(y) \, \mu_g(x) \leq C R^{n-\lambda} \|f\|_{s_1, \lambda}^{s_1}
   \end{align} for some \(C > 0\) depending on the geometry of \(M\). By the heat kernel upper bound (\ref{heatkernel:upperbound}), the LHS of (\ref{Morreyheatestimates:case2step1}) is
   \begin{align*}
      \leq C_u t^{-\frac{n}{2}}\bigg(\underbrace{\int_{B_R(x_0)}\int_{B_{4R}(x_0)}}_{\text{I}} + \underbrace{\int_{B_R(x_0)}\int_{M \setminus B_{4R}(x_0)}}_{\text{II}}\bigg) \, e^{-\frac{d(x, y)^2}{C_ut}} |f(y)|^{s_1} \, \mu_g(y) \, \mu_g(x).
   \end{align*} We now estimate the integrals \(\text{I}\) and \(\text{II}\), starting with \(\text{I}\). Choose geodesic coordinates
   \begin{align*}
       N_{x_0} : B_{6R}(0) \subset \R^n \to B_{6R}(x_0) \subset M.
   \end{align*} For $w \in B_R(0) \subset \R^n,$ denote the set $w - B_{4R}(0) := \{w - z \mid z \in B_{4R}(0)\}.$ By the Euclidean bounds per item (\ref{Morreyheatestimates:uniformeuclidean}),
   \begin{align*}
       \text{I} &\leq 2^n \int_{B_R(0)} \int_{B_{4R}(0)} e^{-\frac{|w-z|^2}{2C_ut}} |f(N_{x_0}(z))|^{s_1} \, \mu_{\text{Euc}}(z) \, \mu_{\text{Euc}}(w) \\
       &= 2^n \int_{B_R(0)} \int_{w-B_{4R}(0)} e^{-\frac{|z|^2}{2C_ut}} |f(N_{x_0}(w-z))|^{s_1} \, \mu_{\text{Euc}}(z) \, \mu_{\text{Euc}}(w) \\
       &\leq 2^n  \int_{B_{5R}(0)} e^{-\frac{|z|^2}{2C_ut}} \int_{B_R(0)}|f(N_{x_0}(z+w))|^{s_1} \, \mu_{\text{Euc}}(w) \, \mu_{\text{Euc}}(z).
   \end{align*} 
   For \(z \in B_{5R}(0)\), let \(N_{N_{x_0}(z)}\) be some choice of geodesic coordinates on the ball of radius \(6R\) centered at \(N_{x_0}(z)\). Set
   \begin{align*}
       F(w) := N_{N_{x_0}(z)}^{-1}(N_{x_0}(z+w)).
   \end{align*} In view of item (\ref{Morreyheatestimates:uniformeuclidean}), 
   \begin{align*}
       N_{x_0}(z + B_R(0)) \subset N_{N_{x_0}(z)}(B_{2R}(0)).
   \end{align*} Hence by the bounds on the Jacobian of the transition between geodesic coordinates from item (\ref{Morreyheatestimates:uniformtransition}),
   \begin{align*}
       \int_{B_R(0)} |f(N_{x_0}(z+w))|^{s_1} \, \mu_{\text{Euc}}(w) &= \int_{B_R(0)} |f(N_{N_{x_0}(z)}(F(w)))|^{s_1} \, \mu_{\text{Euc}}(w) \\
       &\leq C_n \int_{B_{2R}(0)} |f(N_{N_{x_0}(z)}(w))|^{s_1} \, \mu_{\text{Euc}}(w) \\
       &\leq 2^{\frac{n}{2}} C_n \int_{B_{2R}(N_{x_0}(z))} |f(x)|^{s_1} \, \mu_g(x) \\
       &\leq 2^{\frac{n}{2}} C_n^2 R^{n-\lambda} \|f\|_{s_1, \lambda}^{s_1}.
   \end{align*} Since, for a dimensional constant \(C_1 > 0\), we have
   \begin{align*}
       \int_{B_{5R}(0)} e^{-\frac{|z|^2}{2C_u t}} \, \mu_{\text{Euc}}(z) \leq C_1 (C_u t)^{\frac{n}{2}},
   \end{align*}
   we obtain \begin{align}\label{Morreyheatestimates:I}
       \text{I} \leq R^{n-\lambda} (C_u t)^{\frac{n}{2}} 2^{\frac{3n}{2}} C_1 C_n^2 \|f\|_{s_1, \lambda}^{s_1}.
   \end{align} \par
   We now estimate integral \(\text{II}\).  Take a cover of \(M\) as in item (\ref{Morreyheatestimates:uniformcover}). If \(p_i\) is such that \(d(x_0, p_i) \geq 3R\), then the triangle inequality yields for \(x \in B_R(x_0)\) and \(y \in B_R(p_i)\) that
   \begin{align*}
       d(x_0, p_i) \leq 3d(x, y).
   \end{align*} Moreover, \(\{B_R(p_i) \mid d(x_0, p_i) \geq 3R\}\) covers \(M \setminus B_{4R}(x_0)\). 
   Therefore,
   \begin{align*}
       \text{II} &\leq \sum_{\{i \mid d(x_0, p_i) \geq 3R\}} \int_{B_R(x_0)} \int_{B_R(p_i)} e^{-\frac{d(x_0, p_i)^2}{9C_ut}} |f(y)|^{s_1} \, \mu_g(y) \, \mu_g(x).
    \end{align*} Since \(R \leq R_0\), we have by the volume bounds (\ref{preliminaries:volume}) that the RHS in the last line is
    \begin{align*}
        &\leq \sum_{\{i \mid d(x_0, p_i) \geq 3R\}} VR^ne^{-\frac{d(x_0, p_i)^2}{9C_ut}} \int_{B_R(p_i)} |f(y)|^{s_1} \, \mu_g(y) \\
        &\leq R^{n-\lambda}\|f\|_{s_1, \lambda}^{s_1} \sum_{\{i \mid d(x_0, p_i) \geq 3R\}} VR^ne^{-\frac{d(x_0, p_i)^2}{9C_ut}}.
   \end{align*} Now, since \(d(x_0, p_i) \geq 3R\), we have by the triangle inequality that for \(z \in B_R(p_i)\)
   \begin{align*}
       d(x_0, z) \leq 2d(x_0, p_i).
   \end{align*}
   Then, in conjunction with the volume bounds (\ref{preliminaries:volume}) again, we obtain
   \begin{align*}
       \sum_{\{i \mid d(x_0, p_i) \geq 3R\}} R^ne^{-\frac{d(x_0, p_i)^2}{9C_u t}} &\leq V \sum_{\{i \mid d(x_0, p_i) \geq 3R\}} \int_{B_R(p_i)} e^{-\frac{d(x_0, z)^2}{36C_u t}} \, \mu_g(z).
   \end{align*} By the intersection property of the cover as stated in item (\ref{Morreyheatestimates:uniformcover}), the RHS of this last line is
   \begin{align*}
       \leq VC_n \int_M e^{-\frac{d(x_0, z)^2}{36C_u t}} \, \mu_g(z).
   \end{align*}
   Finally, by the heat kernel lower bound (\ref{heatkernel:lowerbound}) and then the stochastic completeness property (\ref{heatkernel:stochasticcompleteness}), we have
   \begin{align*}
       \int_M e^{-\frac{d(x_0, z)^2}{36C_u t}} \, d \mu_g(z) \leq C_l(18 C_u t)^{\frac{n}{2}} \int H(x_0, z, 18 C_u t) \, \mu_g(z) =  C_l(18 C_u t)^{\frac{n}{2}},
   \end{align*} so we deduce that
   \begin{align}\label{Morreyheatestimates:II}
       \text{II} \leq R^{n-\lambda}C_nV^2(18C_u t)^{\frac{n}{2}}C_l \|f\|_{s_1, \lambda}^{s_1}.
   \end{align} Combining the bound (\ref{Morreyheatestimates:I}) for \(\text{I}\) and the bound (\ref{Morreyheatestimates:II}) for \(\text{II}\), we obtain (\ref{Morreyheatestimates:case2step1}). \par 
   We next obtain the bound (\ref{Morreyheatestimates:estimate}) in the proposition statement from (\ref{Morreyheatestimates:case2step1}). First note that the RHS of (\ref{Morreyheatestimates:case2step1}) is finite for \(s_1 \in [1, \infty)\) such that \(f \in M^{s_1, \lambda}\). By (\ref{preliminaries:Morreyembedding}), \(M^{s_1, \lambda} \subset M^{1, \lambda}\), so (\ref{Morreyheatestimates:case2step1}) and (\ref{heatkernel:stochasticcompleteness}) imply that for a.e. \(x \in B_R(x_0)\), \(f(y)\) is integrable with respect to the probability measure \(H(x, y, t) \mu_g(y)\). Thus, by Jensen's inequality
		\begin{align*}
		\int_{B_R(x_0)} |H_t(f)|^{s_1} \leq \int_{B_R(x_0)} \int_M H(x, y, t) |f(y)|^{s_1} \, \mu_g(y) \leq C R^{n-\lambda} \|f\|_{s_1, \lambda}^{s_1}.
		\end{align*}
		Since \(x_0\) and \(R \leq R_0\) were arbitrary, we obtain (\ref{Morreyheatestimates:estimate}). This concludes Case ii. \par 
        \emph{Case iii: \(s_1 \in [1, \infty), s_2 = \infty\).}
		This case can be deduced from \cite[Lemma 4.2]{semilinearivpriem} as follows. First suppose \(0 < t \leq R_0^2\), and recall our assumption from item (\ref{Morreyheatestimates:rhoinfinity}). In the setup of \cite[Lemma 4.2]{semilinearivpriem}, we take the measure \(\mu(y)\) to be \(|f(y)| \, \mu_g(y)\). Then for some constant \(C > 0\), depending only on \(M\),
        \begin{align*}
            |H_t(f)(x)| \leq \int_M H(x, y, t) \, \mu(y)\leq C t^{-\frac{n}{2}} \sup_{z \in M} \int_{B_{\sqrt{t}}(z)} |f(y)| \, \mu_g(y).
        \end{align*} By H{\"o}lder's inequality and (\ref{preliminaries:volume}), this last expression is
        \begin{align*}
            &\leq C t^{-\frac{n}{2} + \frac{n-\lambda}{2s_1}} \sup_{z \in M} \left(t^{\frac{1}{2}(\lambda-n)}\int_{B_{\sqrt{t}}(z)} |f|^{s_1}\right)^{\frac{1}{s_1}} \left(Vt^{\frac{n}{2}} \right)^{\frac{s_1-1}{s_1}} \\
            &\leq  CV t^{-\frac{\lambda}{2s_1}}\|f\|_{s_1, \lambda}.
        \end{align*} Thus, Case iii holds for \(0 < t \leq R_0^2\). \par
        For \(R_0^2 < t \leq 1\), we do the following. By the semigroup property \cite[(7.51)]{grigoryanheatkernelbook} of the heat kernel and Fubini-Tonelli,
        \begin{align*}
            \|H_t(f)\|_\infty &= \left\|H_{R_0^2}(H_{t - R_0^2}(f))\right\|_\infty \\
            &\leq C V(R_0^2)^{-\frac{\lambda}{2s_1}}\left\|H_{t - R_0^2}(f)\right\|_{s_1, \lambda}.
        \end{align*} Since \(t \leq 1\), we may increase \(C\) so that by Case ii this last expression is 
        \begin{align*}
            \leq C t^{-\frac{\lambda}{2s_1}}\|f\|_{s_1, \lambda}.
        \end{align*} Therefore, the proposition holds for Case iii. \par 
        \emph{Case iv: \(1 \leq s_1 < s_2 < \infty\).} This case follows by interpolating Cases ii and iii as follows:
        \begin{align*}
            \|H_t(f)\|_{s_2, \lambda} &\leq \|H_t(f)\|_\infty^{1 - \frac{s_1}{s_2}}\|H_t(f)\|_{s_1, \lambda}^{\frac{s_1}{s_2}} \\
            &\leq \left(C_{(\ref{Morreyheatestimates:estimate})}t^{-\frac{\lambda}{2s_1}}\|f\|_{s_1, \lambda}\right)^{1- \frac{s_1}{s_2}}\left(C_{(\ref{Morreyheatestimates:estimate})} \|f\|_{s_1, \lambda}\right)^{\frac{s_1}{s_2}} \\
            &= C_{(\ref{Morreyheatestimates:estimate})} t^{-\frac{\lambda}{2}\left(\frac{1}{s_1} - \frac{1}{s_2}\right)} \|f\|_{s_1, \lambda},
        \end{align*} as desired.
	\end{proof}
    Now that we have the preceding Morrey space estimates, we can straightforwardly adapt Souplet's result \cite[Prop. 6.1]{semilinearmorrey}, which is the Euclidean analog of Theorem \ref{thm:Morreysupbound}, to the Riemannian setting.
	\begin{proof}[Proof of Theorem \ref{thm:Morreysupbound}]
       We first prove (\ref{Morreysupbound:criticalestimate}) when \(T \leq 1\).  Before proceeding with the estimate, we first remove the linear term in the RHS of (\ref{diffineq}), i.e., defining
       \begin{align*}
           v(t) := e^{-tB} u(t),
       \end{align*} we show that for \(0 \leq t' < t < T\),
       \begin{align}\label{Morreysupbound:subsolutionwithoutlinearterm}
           v(t) \leq H_{t-t'}(v(t')) + C_{(\ref{Morreysupbound:integratingfactorconstant})}\int_{t'}^t H_{t-s}(v^p(s)) \, ds,
       \end{align} where
\begin{align}\label{Morreysupbound:integratingfactorconstant}
      C_{(\ref{Morreysupbound:integratingfactorconstant})} := Ae^{pB}.
  \end{align} To do this under our regularity assumptions on \(u\), we proceed as follows. 
   For \(f : (t', t) \times M \to \R\), denote
\begin{align*}
    I_{\ell} (f) := \int_{t'}^t \int_{t'}^{s_1} \cdots \int_{t'}^{s_{\ell}} f(s_{\ell+1}) \, ds_{\ell+1} \, \dots \, ds_1.
\end{align*} We claim that for \(0 \leq t' < t < T\) and \(\ell \geq 0\),
  \begin{align}\label{Morreysupbound:integratingfactorinductionclaim}
      &H_{t-t'}(u(t')) + \int_{t'}^t H_{t-s}(Au^p(s) + Bu(s)) \, ds  \nonumber \\
      \leq &\left(\sum_{i = 0}^{\ell} \frac{((t-t')B)^i}{i!}\right)\left(H_{t-t'}(u(t')) + \int_{t'}^t H_{t-s}(Au^p(s)) \, ds\right) + B^{\ell+1} I_{\ell} (H_{t-\cdot}(u(\cdot))).
  \end{align} The base case \(\ell = 0\) is clear, as it is just equality. Suppose the claim holds for \(\ell = m\). By (\ref{weaksolution:subsolution}) and positivity of the heat kernel,
  \begin{align*}
      I_m(H_{t-\cdot}(u(\cdot)))  \leq I_m\left((H_{t-\cdot}\left((H_{\cdot-t'}(u(t')) + \int_{t'}^{\cdot} H_{\cdot - s}(Au^p(s) + Bu(s)) \, ds\right)\right).
  \end{align*} By the semigroup property \cite[(7.51)]{grigoryanheatkernelbook} of the heat kernel,
  \begin{align*}
      I_m(H_{t-\cdot}(H_{\cdot-t'}(u(t')))) = H_{t-t'}(u(t')) I_m(1) = \frac{(t-t')^{m+1}}{(m+1)!} H_{t-t'}(u(t')),
  \end{align*} and similarly
  \begin{align*}
      I_m\left(H_{t-\cdot}\left(\int_{t'}^{\cdot} H_{\cdot - s}(Au^p(s) + Bu(s)) \, ds\right)\right) &= I_m\left(\int_{t'}^{\cdot} H_{t - s}(Au^p(s)) \, ds\right) + BI_{m+1}(H_{t-\cdot}(u(\cdot))) \\
      &\leq \frac{(t-t')^{m+1}}{(m+1)!}\int_{t'}^t H_{t-s}(Au^p(s)) \, ds + BI_{m+1}(H_{t-\cdot}(u(\cdot)).
  \end{align*}
   Thus, (\ref{Morreysupbound:integratingfactorinductionclaim}) holds for \(\ell = m+1\), and hence for all \(\ell\) by induction. Since \(u \in C_{\text{loc}}^0([0, T),M^{q, \lambda})\), \(B^{\ell+1} I_{\ell} (H_{t-\cdot}(u(\cdot)))\) goes to zero in \(M^{q, \lambda}\) as \(\ell \to \infty\). Hence, we have a.e.
  \begin{align*}
      u(t) \leq e^{(t-t')B}\left(H_{t-t'}(u(t')) + \int_{t'}^t H_{t-s}(Au^p(s)) \, ds\right).
  \end{align*} Multiplying by \(e^{-tB}\) and using that \(u^p(s) \leq e^{pB}(e^{-Bs}u(s))^p\) since \(T \leq 1\) yield (\ref{Morreysupbound:subsolutionwithoutlinearterm}).
  \par Now that we have removed the linear term in (\ref{weaksolution:subsolution}), the proof of (\ref{Morreysupbound:criticalestimate}) when \(T \leq 1\) follows essentially as in \cite[Prop. 20.25]{superlinearbook} using Prop. \ref{prop:Morreyheatestimates}; we repeat the proof here for clarity.
      Setting
        \begin{align*}
            \beta := \frac{\lambda}{2}\left(\frac{1}{q}-\frac{1}{r}\right) = \frac{1}{p-1}\left(1 - \frac{q}{r}\right),
        \end{align*} we have \(\beta < \frac{1}{p}\) in view of Def. \ref{defn:weaksolution}. Define
        \begin{align*}
            \tau_1 := \sup\{t \in (0, T) \mid s^\beta\|v(s)\|_{r, \lambda} \leq 2C_{(\ref{Morreyheatestimates:estimate})} \delta \ \ \forall \ s \in (0, t)\}.
        \end{align*}  We first show that \(\tau_1 = T\). Note that \(\tau_1 > 0\) since \(s^\beta\|v(s)\|_{r, \lambda} \to 0\) as \(s \searrow 0\) per Def. \ref{defn:weaksolution}. Since \(T \leq 1\), and since \(\frac{r}{p} \geq 1\) by Def. \ref{defn:weaksolution}, we have by Prop. \ref{prop:Morreyheatestimates} that for \(t \in (0, \tau_1)\) and \(s \in (0, t)\)
        \begin{align*}
            \|H_{t-s}(v^p(s))\|_{r,\lambda} &\leq C_{(\ref{Morreyheatestimates:estimate})}(t-s)^{-\frac{\lambda(p-1)}{2r}}\|v^p(s)\|_{\frac{r}{p}, \lambda} \\
            &= C_{(\ref{Morreyheatestimates:estimate})}(t-s)^{-\frac{q}{r}}\|v(s)\|_{r, \lambda}^p \\
            &\leq 2^pC_{(\ref{Morreyheatestimates:estimate})}^{p+1}(t-s)^{-\frac{q}{r}}s^{-\beta p}\delta^p.
        \end{align*} Since by our hypotheses \(q/r, \beta p < 1\), and since \(\beta+1 = \beta p +\frac{q}{r}\), we have for some constant \(\sigma_1 > 0\), depending only on \(p, q,\) and \(r\), that
        \begin{align*}
            t^\beta\int_0^t (t-s)^{-\frac{q}{r}}s^{-\beta p} \, ds = \sigma_1 < \infty,
        \end{align*} independent of \(t\). Therefore, for \(t \in (0, \tau_1)\), (\ref{Morreysupbound:subsolutionwithoutlinearterm}) with \(t' = 0\) implies
        \begin{align*}
            t^\beta \|v(t)\|_{r, \lambda} &\leq t^\beta\|H_t(v(0))\|_{r, \lambda} + t^\beta C_{(\ref{Morreysupbound:integratingfactorconstant})} \int_0^t \|H_{t-s}(v^p(s))\|_{r, \lambda} \, ds \\
            &\leq C_{(\ref{Morreyheatestimates:estimate})}\|v(0)\|_{q, \lambda} + C_{(\ref{Morreysupbound:integratingfactorconstant})}\sigma_12^pC_{(\ref{Morreyheatestimates:estimate})}^{p+1}\delta^p \\
            &\leq C_{(\ref{Morreyheatestimates:estimate})}\delta + C_{(\ref{Morreysupbound:integratingfactorconstant})}\sigma_12^pC_{(\ref{Morreyheatestimates:estimate})}^{p+1}\delta^p. 
        \end{align*} Hence, if
        \begin{align*}
            \delta_0 < \min\left\{1, \left(C_{(\ref{Morreysupbound:integratingfactorconstant})}\sigma_12^{p+1}C_{(\ref{Morreyheatestimates:estimate})}^p\right)^{-\frac{1}{p-1}}\right\} =: \delta_2,
        \end{align*} we have
        \begin{align*}
            t^\beta\|v(t)\|_{r, \lambda} < \frac{3}{2}C_{(\ref{Morreyheatestimates:estimate})}\delta
        \end{align*} for \(t \in (0, \tau_1)\). Therefore, \(\tau_1\) must in fact equal \(T\) by continuity. \par 
        Next set
        \begin{align}\label{Morreysupbound:vinftyconst}
            C_{(\ref{Morreysupbound:vinftyconst})} := 2^{\frac{\lambda}{2r}+ \beta + 2}C_{(\ref{Morreyheatestimates:estimate})}^2,
        \end{align} and define
        \begin{align*}
            \tau_2 := \sup\left\{t \in (0, T) \mid s^{\frac{1}{p-1}} \|v(s)\|_\infty \leq C_{(\ref{Morreysupbound:vinftyconst})} \delta \ \forall \ s \in (0, t)\right\},
        \end{align*} where again \(\tau_2 > 0\) in view of Def. \ref{defn:weaksolution}. We will show that \(\tau_2 = T\), which implies (\ref{Morreysupbound:criticalestimate}) when \(T \leq 1\). By (\ref{Morreysupbound:subsolutionwithoutlinearterm}) with \(t' = \frac{t}{2}\), and by definition of \(\tau_1\), we have for \(t \in (0, \tau_2)\)
        \begin{align*}
            t^{\frac{1}{p-1}}\|v(t)\|_\infty &\leq t^{\frac{1}{p-1}} \left\|H_{t/2}(v(t/2))\right\|_\infty + C_{(\ref{Morreysupbound:integratingfactorconstant})}t^{\frac{1}{p-1}}\int_{\frac{t}{2}}^t \|v(s)\|_\infty^p \, ds \\
            &\leq 2^{\frac{\lambda}{2r}}C_{(\ref{Morreyheatestimates:estimate})}t^{\frac{1}{p-1} - \frac{\lambda}{2r}}\|v(t/2)\|_{r, \lambda} + 2^{\frac{1}{p-1}}C_{(\ref{Morreysupbound:integratingfactorconstant})}C_{(\ref{Morreysupbound:vinftyconst})}^p\delta^p \\
            &\leq 2^{\frac{\lambda}{2r}+ \beta + 1}C_{(\ref{Morreyheatestimates:estimate})}^2\delta + 2^{\frac{1}{p-1}}C_{(\ref{Morreysupbound:integratingfactorconstant})}C_{(\ref{Morreysupbound:vinftyconst})}^p\delta^p.
        \end{align*}
        Hence, if
        \begin{align*}
            \delta_0 < \min\left\{\left(4C_{(\ref{Morreysupbound:integratingfactorconstant})}2^{\frac{1}{p-1}}C_{(\ref{Morreysupbound:vinftyconst})}^{p-1}\right)^{-\frac{1}{p-1}}, \delta_2\right\} =: \delta_3,
        \end{align*} then 
        \begin{align*}
            t^{\frac{1}{p-1}}\|v(t)\|_\infty < \frac{3}{4}C_{(\ref{Morreysupbound:vinftyconst})}\delta,
        \end{align*} so \(\tau_2\) must in fact equal \(T\) by continuity. Therefore, by definition of \(v\), and since \(\max\{q/r, \beta, T\} \leq 1 < p - \frac{2}{n}\), we obtain (\ref{Morreysupbound:criticalestimate}) with \(C_{(\ref{Morreysupbound:criticalestimate})}\) given by
        \begin{align}\label{Morreysupbound:criticalconst}
            C_{(\ref{Morreysupbound:criticalconst})} = e^B2^{\frac{n}{2}+3}C_{(\ref{Morreyheatestimates:estimate})}^2.
        \end{align} This completes the proof of (\ref{Morreysupbound:criticalestimate}) when \(T \leq 1\). \par
        We next show that the \(M^{q, \lambda}\) norm of \(u\) remains small for \(t \in (0, \min\{T, 1\})\). The argument is essentially the same as in Step 2 in the proof of inequality (8.6) in Prop. 8.1 of \cite{semilinearmorrey}. Namely, by definition of \(\tau_1\), which we showed was equal to \(T\), we have
        \begin{align*}
            \|v(t)\|_{q, \lambda} &\leq \|H_t(v(0))\|_{q, \lambda} + C_{(\ref{Morreysupbound:integratingfactorconstant})} \int_0^t \|H_{t-s}(v^p(s))\|_{q, \lambda} \, ds \\
            &\leq C_{(\ref{Morreyheatestimates:estimate})}\|v(0)\|_{q, \lambda} + C_{(\ref{Morreysupbound:integratingfactorconstant})} C_{(\ref{Morreyheatestimates:estimate})} \int_0^t (t-s)^{-\frac{\lambda}{2}\left(\frac{p}{r}-\frac{1}{q}\right)} \|v^p(s)\|_{\frac{r}{p}, \lambda} \, ds \\
            &\leq C_{(\ref{Morreyheatestimates:estimate})}\delta + C_{(\ref{Morreysupbound:integratingfactorconstant})} C_{(\ref{Morreyheatestimates:estimate})}^{p+1}2^p\delta^p \int_0^t (t-s)^{-\frac{\lambda}{2}\left(\frac{p}{r}-\frac{1}{q}\right)} s^{-\beta p} \, ds.
        \end{align*} Note that
        \begin{align*}
            \frac{\lambda}{2}\left(\frac{p}{r}-\frac{1}{q}\right) + \beta p = 1.
        \end{align*} Thus for some constant \(\sigma_2 > 0\), depending only on \(p, q\), and \(r\),
        \begin{align*}
            \int_0^t (t-s)^{-\frac{\lambda}{2}\left(\frac{p}{r}-\frac{1}{q}\right)} s^{-\beta p} \, ds = \sigma_2 <\infty,
        \end{align*} independent of \(t\). Hence, if we further have
        \begin{align*}
            \delta_0 < \min\left\{\left(2^{p+1}\sigma_2C_{(\ref{Morreysupbound:integratingfactorconstant})}C_{(\ref{Morreyheatestimates:estimate})}^p\right)^{-\frac{1}{p-1}}, \delta_3\right\} =: \delta_4,
        \end{align*} then for \(t \in (0, T)\),
        \begin{align}\label{Morreysupbound:Morreystayssmall}
            \|u(t)\|_{M^{q, \lambda}} < C_{(\ref{Morreysupbound:Morreystayssmallconst})}\delta,
                \end{align} where 
        \begin{align}\label{Morreysupbound:Morreystayssmallconst}
            C_{(\ref{Morreysupbound:Morreystayssmallconst})} := 2e^BC_{(\ref{Morreyheatestimates:estimate})}.
        \end{align} Hence, the \(M^{q, \lambda}\) norm of \(u\) remains small for a short time. \par
        We now prove (\ref{Morreysupbound:criticalestimate}) when \(T > 1\). Set
        \begin{align}\label{Morreysupbound:longconst}
            C_{(\ref{Morreysupbound:longconst})} := (C_{(\ref{Morreysupbound:criticalconst})}+1)^{q_c}C_{(\ref{Morreysupbound:Morreystayssmallconst})},
        \end{align} and note that \(C_{(\ref{Morreysupbound:Morreystayssmallconst})} \geq 1\) since \(C_{(\ref{Morreyheatestimates:estimate})} \geq 1\), making \(C_{(\ref{Morreysupbound:longconst})} \geq 1\). Take \begin{align*}
            \delta_0 < C_{(\ref{Morreysupbound:longconst})}^{-1} \delta_4 =: \delta_5.
        \end{align*} In view of (\ref{Morreysupbound:Morreystayssmall}), and since \(C_{(\ref{Morreysupbound:Morreystayssmallconst})} \delta < \delta_3\) by definition of \(\delta_5\), we may apply the argument from the first part of the proof, i.e., the \(T \leq 1\) case of (\ref{Morreysupbound:criticalestimate}), on the interval \([t-1, t]\) for \(t \in (1, \min\{2, T\})\) to obtain
        \begin{align*}
            \|u(t)\|_\infty \leq C_{(\ref{Morreysupbound:criticalconst})}C_{(\ref{Morreysupbound:Morreystayssmallconst})}\delta.
        \end{align*} Thus, (\ref{Morreysupbound:criticalestimate}) holds on \((0, \min\{2, T\})\) with \(C_{(\ref{Morreysupbound:criticalestimate})} = C_{(\ref{Morreysupbound:longconst})}\). \par 
        Next, if \(T > 2\), we define
        \begin{align}
            \tau_3 := \sup\{t \in (2, T) \mid \|u(s)\|_\infty \leq C_{(\ref{Morreysupbound:longconst})}\delta \ \forall \ s \in (1, t)\},
        \end{align} where \(\tau_3 > 2\) by continuity. We will show that \(\tau_3 = T\), which implies (\ref{Morreysupbound:criticalestimate}) when \(T > 2\). 
         For \(t \in (1, \tau_3)\), we have by (\ref{Morreysupbound:criticalinitialhypothesis}) that
        \begin{align*}
            \sup_{x \in M} \|u(t)\|_{L^{q_c}(B_1(x))} &\leq \sup_{x \in M} \|u(t)\|_{L^{\infty}(B_1(x))}^{1-\frac{s}{q_c}}\|u(t)\|_{L^s(B_1(x))}^{\frac{s}{q_c}} \\
            &< C_{(\ref{Morreysupbound:longconst})}^{1-\frac{s}{q_c}}\delta.
        \end{align*} Let \(\delta_3'\) denote the choice of \(\delta_3\) for \(q = q_c\), and take
        \begin{align*}
            \delta_0 < \min\{C_{(\ref{Morreysupbound:longconst})}^{-1}\delta_3', \delta_5\} =: \delta_6.
        \end{align*} Since \(C_{(\ref{Morreysupbound:longconst})} \geq 1\), we therefore have
        \begin{align*}
            \sup_{x \in M} \|u(t)\|_{L^{q_c}(B_1(x))} < \delta_3',
        \end{align*} so we may apply the argument from the first part of the proof, with \(q = q_c\), on the interval \([t - 1, t]\) for \(t \in (2, \tau_3)\), yielding
        \begin{align*}
            \|u(t)\|_\infty \leq C_{(\ref{Morreysupbound:criticalconst})}C_{(\ref{Morreysupbound:longconst})}^{1-\frac{s}{q_c}}\delta.
        \end{align*} By construction,
        \begin{align*}
            C_{(\ref{Morreysupbound:criticalconst})}C_{(\ref{Morreysupbound:longconst})}^{1-\frac{s}{q_c}} < C_{(\ref{Morreysupbound:longconst})},
        \end{align*} so \(\tau_3 = T\) by continuity and we take 
        \begin{align*}
            C_{(\ref{Morreysupbound:criticalestimate})} = C_{(\ref{Morreysupbound:longconst})}.
        \end{align*} This concludes the proof of (\ref{Morreysupbound:criticalestimate}). \par 
        Finally, we prove (\ref{Morreysupbound:improvedestimate}). Recall that now \(u(0)\) additionally is in \(M^{q', \lambda'}\), where 
        \begin{align*}
            \alpha := \frac{\lambda'}{2q'} < \frac{1}{p-1}.
        \end{align*} In view of (\ref{Morreysupbound:criticalestimate}), we just need to show that \(\|u(t)\|_\infty \leq Ct^{-\alpha}\) for \(t \in (0, \min\{1, T\})\), for some suitable constant \(C > 0\), which we now do. \par 
        By (\ref{Morreysupbound:criticalestimate}) and (\ref{Morreysupbound:subsolutionwithoutlinearterm}), \(v\) is a solution of the differential inequality
        \begin{align*}
            \left(\frac{\partial}{\partial t} - \Delta\right) v \leq \delta't^{-1} v
        \end{align*} for \(t \in (0, \min\{1, T\})\), where
        \begin{align*}
            \delta' := C_{(\ref{Morreysupbound:integratingfactorconstant})} C_{(\ref{Morreysupbound:vinftyconst})}^{p-1}\delta^{p-1}.
        \end{align*}Define 
        \begin{align*}
             \nu &:= \alpha(p-1) < 1 \\
             T' &:= \min\left\{1, T, \left(C_{(\ref{Morreysupbound:integratingfactorconstant})}4^{p+1}C_{(\ref{Morreyheatestimates:estimate})}^{2p-1}K^{p-1} (1-\nu)^{-1}\right)^{\frac{-1}{1-\nu}}\right\} \\
            \tau_4 &:= \sup\left\{t \in (0, T') \mid \|v(s)\|_{q', \lambda'} \leq 2C_{(\ref{Morreyheatestimates:estimate})} K \ \forall \ s \in (0, t)\right\} \\
            \tau_5 &:= \sup\left\{t \in (0, T') \mid s^{\alpha}\|v(s)\|_\infty \leq 2^{\alpha + 2}C_{(\ref{Morreyheatestimates:estimate})}^2 K \ \forall \ s \in (0, t)\right\}.
        \end{align*} We first prove (\ref{Morreysupbound:improvedestimate}) on \((0, T')\). Note that \(\tau_4, \tau_5 > 0\) by the assumptions \(C_{(\ref{Morreyheatestimates:estimate})} \geq 1\) and (\ref{Morreysupbound:improvedauxi}).  Suppose \(\tau_5 < \tau_4\). Then for \(0 < t < \tau_5\)
        \begin{align*}
            \|v(t)\|_\infty &\leq \left\|H_{t/2}(v(t/2))\right\|_\infty + \delta'\int_{\frac{t}{2}}^t s^{-1}\|H_{t-s}(v(s))\|_\infty \, ds \\
             &\leq  C_{(\ref{Morreyheatestimates:estimate})}(t/2)^{-\frac{\lambda'}{2q'}} \|v(t/2)\|_{q', \lambda'} + \delta'\int_{\frac{t}{2}}^t s^{-1} \|v(s)\|_\infty \, ds \\
            &\leq 2^{\alpha+1} C_{(\ref{Morreyheatestimates:estimate})}^2 t^{-\alpha}K + \delta'2^{\alpha + 2}C_{(\ref{Morreyheatestimates:estimate})}^2 K\int_{\frac{t}{2}}^t s^{-1-\alpha} \, ds  \\
            &\leq t^{-\alpha}\left(2^{\alpha+1} C_{(\ref{Morreyheatestimates:estimate})}^2 K + \delta'2^{\alpha + 2}C_{(\ref{Morreyheatestimates:estimate})}^2K\alpha^{-1}(2^\alpha-1)\right).
        \end{align*} If
        \begin{align*}
            \delta_0 < \min\left\{\left(4C_{(\ref{Morreysupbound:integratingfactorconstant})}C_{(\ref{Morreysupbound:vinftyconst})}^{p-1}\alpha^{-1}(2^\alpha-1)\right)^{-\frac{1}{p-1}}, \delta_6\right\} =: \delta_7,
        \end{align*} then
        \begin{align*}
            t^\alpha\|v(t)\|_\infty \leq \frac{3}{2}2^{\alpha + 1}C_{(\ref{Morreyheatestimates:estimate})}^2 K < 2^{\alpha + 2}C_{(\ref{Morreyheatestimates:estimate})}^2 K, \ \ \ \ (\ast)
        \end{align*} so in fact \(\tau_5 \geq \tau_4\) by continuity. (Since \(p > 1 + \frac{2}{n}\), we have \(\alpha \leq \frac{n}{2}\), so the dependencies of \(\delta_0\) have not changed.) \par 
        Now, suppose \(\tau_4 < \tau_5\). Observe that on \((0, \tau_5)\), \(v\) is a solution of the differential inequality
        \begin{align*}
            \left(\frac{\partial}{\partial t} - \Delta\right) v \leq C_{(\ref{Morreysupbound:integratingfactorconstant})} (2^{\alpha + 2}C_{(\ref{Morreyheatestimates:estimate})}^2K)^{p-1}t^{-\nu} v.
        \end{align*} Then for \(0 < t < \tau_4\)
        \begin{align*}
            \|v(t)\|_{q', \lambda'} &\leq \|H_t(v(0))\|_{q', \lambda'} + C_{(\ref{Morreysupbound:integratingfactorconstant})}(2^{\alpha + 2}C_{(\ref{Morreyheatestimates:estimate})}^2K)^{p-1}\int_0^t s^{-\nu}\|H_t(v(s))\|_{q', \lambda'} \, ds \\
            &\leq C_{(\ref{Morreyheatestimates:estimate})} \|v(0)\|_{q', \lambda'} + C_{(\ref{Morreysupbound:integratingfactorconstant})}2^{\nu + 2(p-1)}C_{(\ref{Morreyheatestimates:estimate})}^{2p-1}K^{p-1}\int_0^t s^{-\nu}\|v(s)\|_{q', \lambda'} \,ds \\
            &\leq C_{(\ref{Morreyheatestimates:estimate})}K + C_{(\ref{Morreysupbound:integratingfactorconstant})}2^{2p}C_{(\ref{Morreyheatestimates:estimate})}^{2p}K^{p} (1-\nu)^{-1}t^{1-\nu}.
        \end{align*} Since \(t \leq T'\), we have
        \begin{align*}
            \|v(t)\|_{q', \lambda'} \leq \frac{3}{2}K < 2K, \ \ \ \ (\ast\ast)
        \end{align*} so in fact \(\tau_4 \geq \tau_5\) by continuity. Thus, \(\tau_4 = \tau_5\), and therefore, \(\tau_4 = \tau_5 = T'\) by strictness of the inequalities $(\ast)$ and $(\ast\ast).$ Thus, (\ref{Morreysupbound:improvedestimate}) holds on \((0, T')\). \par
        Lastly, we prove (\ref{Morreysupbound:improvedestimate}) for \(t \in (T', \min\{1, T\})\). By (\ref{Morreysupbound:criticalestimate}),
        \begin{align*}
            \|v(t)\|_\infty &\leq \|H_{t-T'}(v(T'))\|_\infty + C_{(\ref{Morreysupbound:integratingfactorconstant})} \int_{T'}^t \|H_{t-s}(v^p(s))\|_\infty \, ds \\
            &\leq \|v(T')\|_\infty + C_{(\ref{Morreysupbound:integratingfactorconstant})} \int_{T'}^t \|v(s)\|_\infty^p \, ds \\
            &\leq C_4KT'^{-\alpha} + C_{(\ref{Morreysupbound:integratingfactorconstant})} (C_{(\ref{Morreysupbound:vinftyconst})}\delta)^p\int_{T'}^t s^{-\frac{p}{p-1}} \, ds \\
            &\leq C_4KT'^{-\alpha} + C_{(\ref{Morreysupbound:integratingfactorconstant})} (p-1)\left(T'^{-\frac{1}{p-1}}-1\right),
        \end{align*} so for \(t \in (0, \min\{1, T\})\)
        \begin{align}
            \|u(t)\|_\infty \leq C_{(\ref{Morreysupbound:improvedconst})}t^{-\alpha},
        \end{align} where
        \begin{align}\label{Morreysupbound:improvedconst}
            C_{(\ref{Morreysupbound:improvedconst})} = e^B\left(C_4KT'^{-\alpha} + C_{(\ref{Morreysupbound:integratingfactorconstant})} (p-1)\left(T'^{-\frac{1}{p-1}}-1\right)\right).
        \end{align} The bound (\ref{Morreysupbound:improvedestimate}) now follows, including for \(T > 1\), in view of (\ref{Morreysupbound:criticalestimate}).
	\end{proof}

\section{Application to harmonic map flow}
We now use the preceding Morrey estimates to prove Theorem \ref{thm:regularballtheorem} and Corollary \ref{gapcor}. First we will establish that for a suitable class of solutions of 
harmonic map flow, the energy density is a solution of (\ref{diffineq}) (in the sense of Def. \ref{defn:weaksolution}). \par
Let \(M\) and \(N\) be closed Riemannian manifolds, with \(M\) having dimension \(n\). By Nash's embedding theorem, we fix an isometric embedding of \(N\) into some Euclidean space \(\R^k\). Sobolev maps between \(M\) and \(N\) are then defined to be Sobolev \(\R^k\)-valued functions on \(M\) taking value a.e. in \(N \subset \R^k\). Letting \(\Pi\) denote the second fundamental form of this embedding, the harmonic map flow equation is equivalent to
    \begin{align*}
        \left(\frac{\partial}{\partial t} - \Delta\right) w = \Pi(w)(\nabla w, \nabla w),
    \end{align*} where \(w(t)\) is a family of \(\R^k\)-valued functions on \(M\) whose image is in \(N\) a.e., and \(\nabla\) and \(\Delta\), respectively, denote the gradient and Laplacian of \(M\) on \(\R^k\)-valued functions. By work of Wang \cite{hmfbmo}, we know that given an initial map \(w_0 \in W^{1, n}(M, N)\), there exists a solution of harmonic map flow in the following sense. By abuse of notation, we let \(\Pi\) denote the smooth, compactly supported extension of the second fundamental form 
    as in \cite[pg. 11]{hmfbmo}. For \(T \in (0, \infty)\), consider the Banach space \(X_T\) of functions \(f : M \times [0, T] \to \R^k\) defined by the norm
   \begin{align*}
       ||| f|||_{X_T} := \sup_{0 \leq t \leq T} \|f\|_\infty + \sup_{0 < t \leq T} \sqrt{t} \|\nabla f\|_\infty + \sup_{x \in M, 0 < R \leq \sqrt{T}} \left(R^{-n} \int_{B_R(x) \times [0, R^2]} |\nabla f|^2 \right)^{\frac{1}{2}}.
   \end{align*} Wang showed that for some \(T > 0\), depending on \(w_0\) and which we may take to be less than \(R_0^2\), the operator 
   \begin{align*}
       \Phi(f)(t) := H_t(w_0) + \int_0^t H_{t-s}(\Pi(f(s))(\nabla f(s), \nabla f(s))) \, ds
   \end{align*} has a unique fixed point \(w\) in some small ball about \(H_t(w_0)\) in \(X_T\). Additionally, for \((x, t) \in M \times (0, T]\), \(w\) is smooth and takes value in \(N\). \par
   We now explain the following regularity properties of \(w\) near the initial time in Lemma \ref{lemma:checkinghmfisMorreysolution}. We follow the strategy of Brezis and Cazenave \cite[Proof of Theorem 1, pg. 284-286]{breziscazenavesemilinear}.
   First recall the following gradient estimates for heat kernel on \(M\), where \(f_0\) is an \(\R^k\)-valued function.
   \begin{lemma}\label{lemma:gradientestimates}
        Let \(0 < t \leq 1\) and \(1 \leq q \leq r \leq \infty\). Then, there exists \(C \geq 1\), depending only on the geometry of \(M\), such that
   \begin{align*}
       &\|\nabla H_t(f_0)\|_r \leq C t^{-\frac{n}{2}\left(\frac{1}{q} -\frac{1}{r}\right)} \cdot \begin{cases} t^{-\frac12}\|f_0\|_q & f_0 \in L^q\\
       \| \nabla f_0\|_q & f_0 \in W^{1,q}.
       \end{cases}
   \end{align*}
   \end{lemma}
   \begin{proof}
       This is well known; see the appendix.
   \end{proof}
\begin{lemma}\label{lemma:checkinghmfisMorreysolution}
    We have \(w \in C^0([0, T], W^{1, n}) \cap C_{\text{loc}}^0((0, T], W^{1, \infty})\). Moreover, 
    \begin{align*}
        \lim_{t \searrow 0} \left(\|w(t) - w_0\|_{W^{1, n}} + t^{\frac{1}{6}}\|\nabla w(t)\|_{\frac{3n}{2}} + t^{\frac{1}{2}}\|\nabla w(t)\|_{\infty}\right) = 0.
    \end{align*} In particular, \(|\nabla w|\) is a solution of (\ref{diffineq}) with \(p = 3, q = 2\), and \(r = 3\) as in Def. \ref{defn:weaksolution}.
\end{lemma}
\begin{proof}
   
   Since \(w(t)\) is smooth for \(t > 0\) and since \(M\) is closed, we have
   \begin{align*}
       w \in C_{\text{loc}}^0((0, T], W^{1, \infty}) \subset C_{\text{loc}}^0((0, T], W^{1, n}).
   \end{align*} Thus, we just show that the limit in the statement of this lemma is zero. We establish this in the following order:
   \begin{enumerate}
       \item \(\lim_{t \searrow 0} t^{\frac{1}{6}} \|\nabla w(t)\|_{\frac{3n}{2}} = 0\); 
       \item \(\lim_{t \searrow 0} t^{\frac{1}{2}} \|\nabla w(t)\|_\infty = 0\); 
       \item \(\lim_{t \searrow 0} \|w(t) - w_0\|_{W^{1, n}} = 0\).
   \end{enumerate}\par
   In the sequel, \(C\) denotes a positive constant depending on the geometries of \(M\) and \(N\) which may increase from line to line. \par
   \emph{Item 1: \(\lim_{t \searrow 0} t^{\frac{1}{6}} \|\nabla w(t)\|_{\frac{3n}{2}} = 0\).}
   Denote
   \begin{align*}
       w_1(t) := H_t(w_0).
   \end{align*} We first claim that
   \begin{align}\label{checkinghmfisMorreysolution:goestozero}
      \lim_{t \searrow 0} \left(t^{\frac{1}{6}} \|\nabla w_1(t)\|_{\frac{3n}{2}} + t^{\frac{1}{2}}\|\nabla w_1(t)\|_\infty\right) = 0. 
   \end{align} To see this, first note that since \(w_0 \in W^{1, n}\) and since \(M\) is closed, there exists a sequence \(\tilde{w}_i \in C^\infty(M, \R^k)\) converging to \(w_0\) in \(W^{1, n}(M, \R^k)\). Then by Lemma \ref{lemma:gradientestimates},
    \begin{align*}
        \|\nabla w_1(t)\|_{\frac{3n}{2}} &\leq \|\nabla w_1(t) - \nabla H_t(\tilde{w}_i)\|_{\frac{3n}{2}} + \|\nabla H_t(\tilde{w}_i)\|_{\frac{3n}{2}} \\
        &\leq C t^{-\frac{1}{6}}\|\nabla w_0 - \nabla \tilde{w}_i\|_n + C \|\nabla H_t(\tilde{w}_i)\|_\infty \\
        &\leq C t^{-\frac{1}{6}}\|\nabla w_0 - \nabla \tilde{w}_i\|_n + C \|\nabla \tilde{w}_i\|_\infty.
    \end{align*} We multiply by \(t^{\frac{1}{6}}\). Choosing \(i\) and then \(t\) appropriately, we deduce that \(\lim_{t \searrow 0} t^{\frac{1}{6}} \|\nabla w_1(t)\|_{\frac{3n}{2}} = 0\). An analogous estimation implies that \(\lim_{t \searrow 0} t^{\frac{1}{2}} \|\nabla w_1(t)\|_{\infty} = 0\). Thus, the claim (\ref{checkinghmfisMorreysolution:goestozero}) holds. \par
    We now use (\ref{checkinghmfisMorreysolution:goestozero}) to prove Item 1. By Wang's construction, \(w = \lim_{n \to \infty} \Phi(w_n)\) in \(X_T\), where \(w_1\) is as before and \(w_n = \Phi(w_{n-1})\).
    Set
    \begin{align}\label{checkinghmfisMorreysolution:duhamelconstant}
        C_{(\ref{checkinghmfisMorreysolution:duhamelconstant})} := \int_0^1 (1-\sigma)^{-\frac{5}{6}} \sigma^{-\frac{2}{3}} \, d\sigma < \infty.
    \end{align} 
    Recall that we have assumed \(T \leq 1\). Then for all \(i\) and each \(\tau \in (0, T]\), Lemma \ref{lemma:gradientestimates} yields
    \begin{align*}
        \|\nabla w_{i+1}(\tau)\|_{\frac{3n}{2}} - \|\nabla w_1(\tau)\|_{\frac{3n}{2}} &\leq  \int_0^\tau \|\nabla H_{\tau-s}( \Pi(w_i(s))(\nabla w_i(s), \nabla w_i(s)))\|_{\frac{3n}{2}} \, ds \\ 
        &\leq C\int_0^\tau (\tau-s)^{-\frac{5}{6}}\|\Pi(w_i(s))(\nabla w_i(s), \nabla w_i(s))\|_{\frac{3n}{4}} \, ds \\ 
        &\leq  C\int_0^\tau (\tau-s)^{-\frac{5}{6}}\|\nabla w_i(s)\|_{\frac{3n}{2}}^2 \, ds \\ 
        &\leq  C\left(\sup_{0 < s \leq \tau} s^{\frac{1}{6}} \|\nabla w_i(s)\|_{\frac{3n}{2}}\right)^2\int_0^\tau (\tau-s)^{-\frac{5}{6}}s^{-\frac{1}{3}} \, ds \\ 
        &\leq  C\left(\sup_{0 < s \leq \tau} s^{\frac{1}{6}} \|\nabla w_i(s)\|_{\frac{3n}{2}}\right)^2C_{(\ref{checkinghmfisMorreysolution:duhamelconstant})} \tau^{-\frac{1}{6}}.
    \end{align*} Hence,
    \begin{align*}
        \tau^{\frac{1}{6}} \|\nabla w_{i+1}(\tau)\|_{\frac{3n}{2}} \leq \sup_{0 < s \leq \tau} s^{\frac{1}{6}}\|\nabla w_1(s)\|_{\frac{3n}{2}} + CC_{(\ref{checkinghmfisMorreysolution:duhamelconstant})}\left(\sup_{0 < s \leq \tau} s^{\frac{1}{6}} \|\nabla w_i(s)\|_{\frac{3n}{2}}\right)^2.
    \end{align*} Fix \(t \in (0, T]\). The preceding inequality holds for all \(\tau \in (0, t]\), so we replace \(\sup_{0 < s \leq \tau}\) in the RHS by \(\sup_{0 < s \leq t}\) and then supremize the LHS over \(\tau \in (0, t]\) to obtain
    \begin{align}\label{checkinghmfisMorreysolution:item1inductioninequality}
        \sup_{0 < s \leq t}s^{\frac{1}{6}} \|\nabla w_{i+1}(s)\|_{\frac{3n}{2}} \leq \sup_{0 < s \leq t} s^{\frac{1}{6}}\|\nabla w_1(s)\|_{\frac{3n}{2}} + CC_{(\ref{checkinghmfisMorreysolution:duhamelconstant})}\left(\sup_{0 < s \leq t} s^{\frac{1}{6}} \|\nabla w_i(s)\|_{\frac{3n}{2}}\right)^2.
    \end{align} We now use (\ref{checkinghmfisMorreysolution:item1inductioninequality}) and induction to show that \(\sup_{0 < s \leq t} s^{\frac{1}{6}} \|\nabla w_i(s)\|_{\frac{3n}{2}}\) can be made small for all \(i\) if \(t\) is chosen small enough. \par 
    Let
    \begin{align*}
        \varepsilon \in \left(0, (4CC_{(\ref{checkinghmfisMorreysolution:duhamelconstant})})^{-1}\right).
    \end{align*}(\ref{checkinghmfisMorreysolution:goestozero}) implies there exists \(t > 0\) such that 
    \begin{align*}
        \sup_{0 < s \leq t} s^{\frac{1}{6}}\|\nabla w_1(s)\|_{\frac{3n}{2}} \leq \varepsilon.
    \end{align*} Denote
    \begin{align*}
        \mu_i := \sup_{0 < s \leq t} s^{\frac{1}{6}} \|\nabla w_i(s)\|_{\frac{3n}{2}}.
    \end{align*} By choice of \(t\), we have
    \begin{align*}
        \mu_1 \leq \varepsilon.
    \end{align*} Suppose that we have 
    \begin{align*}
        \mu_i \leq 2\varepsilon.
    \end{align*} Then by (\ref{checkinghmfisMorreysolution:item1inductioninequality})
    \begin{align*}
        \mu_{i+1} \leq \varepsilon + CC_{(\ref{checkinghmfisMorreysolution:duhamelconstant})}\mu_i^2 \leq \varepsilon + 4CC_{(\ref{checkinghmfisMorreysolution:duhamelconstant})}\varepsilon^2 \leq 2\varepsilon.
    \end{align*} Thus by induction, \(\mu_i \leq 2\varepsilon\) for all \(i\).  From the definition of the \(X_T\) norm and the convergence \(w_i \to w\) in \(X_T\), we have that \(\nabla w_i \to \nabla w\) in \(C_{\text{loc}}^0((0, T), L^{\infty})\). Hence, 
    \begin{align*}
        \sup_{0 < s \leq t} s^{\frac{1}{6}}\|\nabla w(s)\|_{\frac{3n}{2}}\leq 2\varepsilon,
    \end{align*} so Item 1 follows. \par
    \emph{Item 2: \(\lim_{t \searrow 0} t^{\frac{1}{2}}\|\nabla w(t)\|_\infty = 0\).}  Since \(w\) solves the fixed-point equation, we have by Lemma \ref{lemma:gradientestimates}
    \begin{align*}
        \|\nabla w(t)\|_\infty - \|\nabla w_1(t)\|_\infty  &= \|\nabla \Phi(w)(t)\|_\infty - \|\nabla w_1(t)\|_\infty \\
        &\leq C\int_0^t (t-s)^{-\frac{5}{6}} \|\Pi(w(s))(\nabla w(s), \nabla w(s)) \|_{\frac{3n}{2}} \, ds \\
        &\leq C\left(\sup_{0 < s \leq t} s^{\frac{1}{2}} \|\nabla w(s)\|_\infty\right)\left(\sup_{0 < s \leq t} s^{\frac{1}{6}} \|\nabla w(s)\|_{\frac{3n}{2}}\right)\int_0^t (t-s)^{-\frac{5}{6}} s^{-\frac{2}{3}} \, ds \\
        &\leq C\left(\sup_{0 < s \leq t} s^{\frac{1}{2}} \|\nabla w(s)\|_\infty\right)\left(\sup_{0 < s \leq t} s^{\frac{1}{6}} \|\nabla w(s)\|_{\frac{3n}{2}}\right)C_{(\ref{checkinghmfisMorreysolution:duhamelconstant})}t^{-\frac{1}{2}}. 
    \end{align*} We multiply by \(t^{\frac{1}{2}}\). By definition of the \(X_T\) norm, we know that \(\sup_{0 < s \leq T} s^{\frac{1}{2}}\|\nabla w(s)\|_\infty < \infty\). Then, since by Item 1 and (\ref{checkinghmfisMorreysolution:goestozero})
    \begin{align*}
        \lim_{t \searrow 0} \left(t^{\frac{1}{6}}\|\nabla w(t)\|_{\frac{3n}{2}} + t^{\frac{1}{2}}\|\nabla w_1(t)\|_\infty\right) = 0,
    \end{align*} Item 2 follows. \par
    \emph{Item 3: \(\lim_{t \searrow 0} \|w(t) - w_0\|_{W^{1, n}} = 0\).} Since \(w_0 \in W^{1, n}\), \(w_1(t) \to w_0\) in \(W^{1, n}(M, \R^k)\) as \(t \searrow 0\). Next, we estimate
    \begin{align*}
        \left\|\int_0^t H_{t-s}(\Pi(w)(\nabla w(s), \nabla w(s))) \, ds\right\|_n &\leq C\int_0^t (t-s)^{-\frac{1}{3}} \|\Pi(w)(\nabla w(s), \nabla w(s))\|_{\frac{3n}{4}} \, ds \\
        &\leq C \left(\sup_{0 < s \leq t} s^{\frac{1}{6}} \|\nabla w(s)\|_{\frac{3n}{2}}\right)^2 \int_0^t (t-s)^{-\frac{1}{3}}s^{-\frac{1}{3}} \, ds \\
        &\leq CC_{(\ref{checkinghmfisMorreysolution:duhamelconstant})} \left(\sup_{0 < s \leq t} s^{\frac{1}{6}} \|\nabla w(s)\|_{\frac{3n}{2}}\right)^2
    \end{align*} and
    \begin{align*}
        \left\|\int_0^t \nabla H_{t-s}(\Pi(w)(\nabla w(s), \nabla w(s))) \, ds\right\|_n &\leq \int_0^t C(t-s)^{-\frac{2}{3}} \|\Pi(w)(\nabla w(s), \nabla w(s))\|_{\frac{3n}{4}} \, ds \\
        &\leq CC_{(\ref{checkinghmfisMorreysolution:duhamelconstant})}  \left(\sup_{0 < s \leq t} s^{\frac{1}{6}} \|\nabla w(s)\|_{\frac{3n}{2}}\right)^2.
    \end{align*} Since \(w\) solves the fixed-point equation, we deduce that 
    \begin{align*}
        \|w(t) - w_0\|_{W^{1, n}} &= \|\Phi(w)(t) - w_0\|_{W^{1, n}} \\
        &\leq \|w_1(t) - w_0\|_{W^{1, n}} + \left\|\int_0^t  H_{t-s}(\Pi(w(s))(\nabla w(s), \nabla w(s))) \,ds \right\|_{W^{1, n}} \\
        &\to 0 \ \text{ as } \ t \searrow 0,
    \end{align*} as desired.
    \par
    We have thus shown that the limit in the statement of the lemma is zero. It remains to deduce that
    \begin{align*}
        u := |\nabla w|
    \end{align*}is a solution of (\ref{diffineq}), which we now explain. It follows from H{\"o}lder's inequality that
    \begin{align*}
        u \in C^0([0, T], M^{2, 2}) \cap  C_{\text{loc}}^0((0, T], L^\infty),
    \end{align*} and furthermore
    \begin{align*}
       \lim_{t \searrow 0} \left(t^{\frac{1}{6}}\|u\|_{M^{3, 2}} + t^{\frac{1}{2}}\|u\|_\infty\right) = 0. 
    \end{align*} Since \(w\) is smooth on \(M \times (0, T]\), the B{\"o}chner identity for harmonic map flow and the comparison principle yields for all \(0 < t' < t \leq T\) 
    \begin{align}\label{checkinghmfisMorreysolution:subsolution}
        u(t) \leq H_{t-t'}(u(t')) + \int_{t'}^t H_{t-s}(Au^3(s) + Bu(s)) \, ds.
    \end{align} Note that for fixed \(t > 0\), we have that as \(t' \searrow 0\)
    \begin{align*}
        \| H_{t-t'}(u(t')) - H_t(u(0))\|_n &= \|H_{t-t'}(u(t')-H_{t'}(u(0)))\|_n \\
        &\leq C\|u(t') - H_{t'}(u(0))\|_n \\
        &\leq C(\|u(t') - u(0)\|_n + \|u(0) - H_{t'}(u(0))\|_n) \\
        &\to 0
    \end{align*} and 
    \begin{align*}
        \left\|\int_0^{t'} \hspace{-0.4em} H_{t-s} (Au^3(s) + Bu(s)) \, ds\right\|_n &\leq CA\hspace{-0.4em}\int_0^{t'} \hspace{-0.4em} (t-s)^{-1}\|u^3(s)\|_{\frac{n}{3}} \, ds + CB\hspace{-0.4em}\int_0^{t'} \hspace{-0.4em}\|u(s)\|_n \, ds \\
        &\leq CA(\|u(0)\|_n+1)^3\ln\left(\frac{t}{t-t'}\right) + CB(\|u(0)\|_n+1)t' \\
        &\to 0.
    \end{align*} Therefore, (\ref{checkinghmfisMorreysolution:subsolution}) also holds with \(t' = 0\). Hence, we conclude that \(|\nabla w|\) is a solution of (\ref{diffineq}).
\end{proof}

We will also make use of the following standard lemma, which can be compared with \cite[Thm. 1.4]{bolingkelleherstreets}, as well as the earlier work \cite[Cor. 1.1]{chending}.
\begin{lemma}\label{lemma:regularballlemma}
    Suppose \(w : M \times [0, T) \to N\) is a classical solution of harmonic map flow between closed Riemannian manifolds \((M, g), (N, h)\). If the image of \(w(0)\) is contained in a regular ball \(B \subset N\) (see \cite{Harmonicmapsintoregballs} for a definition), then \(w\) extends to a global classical solution and converges exponentially in \(C^k\) for all \(k\) to a constant map as \(t \to \infty\). 
\end{lemma}
\begin{proof}
   By compactness and continuity, \(w\) is contained in \(B\) for short time. Let \(|w|^2\) denote \(\delta_{ij}w^iw^j\) in some choice of geodesic coordinates on \(B\). From \cite[(2.22)]{Harmonicmapsintoregballs}, we see that \(\left(\frac{\partial}{\partial t} - \Delta\right)|w|^2 \leq 0\), so by the maximum principle we deduce that \(w\) must be contained in \(B\) as long as \(w\) is defined. Then, \cite[Theorem 2.2]{HMFintoregballs} implies that \(w\) must be defined for all time, and we have the bound
\begin{align}\label{regularballlemma:derivativebound}
    \sup_{M \times [1, \infty)} |dw| < \infty.
\end{align} Let \(t_j := j\). Then, \cite[Theorem 4.1]{HMFintoregballs} implies that for a subsequence (also labeled \(t_j\)), \(w( t_j)\) \(C^2\) converges to a harmonic map \(w_\infty\), which must be contained in \(B\). By  \cite[Theorem 1]{Harmonicmapsintoregballs}, \(w_\infty\) must be a constant map, whose image we denote by \(x_\infty\). \par
To show that in fact we have exponential convergence in time to \(w_\infty\), we proceed as follows. For \(\varepsilon > 0\) small enough, \(B' := B_\varepsilon(x_\infty)\) is a regular ball. We fix some choice of normal coordinates on \(B'\). We first show that there exists \(C > 0\) such that for \(t\) large enough
    \begin{align}\label{regularballlemma:tensionboundsenergy}
        \int |dw|^2 \leq C \int |\tau(w)|^2.
    \end{align} Here \(\tau(w)\) denotes the tension field, i.e., map Laplacian, of \(w\), and \(C\) may increase from line to line in the argmument.  Write \(w = w^i e_i\) in the coordinates on \(B'\), and set
    \begin{align*}
        c^i := \frac{1}{\text{Vol}(M)}
        \int_M w^i.
    \end{align*}In the coordinates on \(B'\)
    \begin{align*}
        C^{-1}\delta_{ij} \leq h_{ij} \leq C \delta_{ij}.
    \end{align*}We compute
    \begin{align*}
        \int |dw|^2 &\leq C \sum_i \int |dw^i|^2 \\
        &= C \sum_i \int \langle w^i - c^i, d^\ast d w^i\rangle \\
        &\leq C \sum_i \sqrt{\int |w^i - c^i|^2} \sqrt{\int |d^\ast d w^i|^2} \\
        &\leq C \sum_i \sqrt{\int |dw^i|^2} \sqrt{\int |d^\ast d w^i|^2}.
    \end{align*} Hence,
    \begin{align*}
        \int |dw|^2 \leq C 
        \sum_i \int |d^\ast dw^i|^2.
    \end{align*} Let \(\Gamma\) denote the Christoffel symbols of \(h\) on \(B'\). Then,
    \begin{align*}
        \sum_i d^\ast d w^i = \tau(w) - \Gamma(w)(\text{tr}_g(dw \otimes dw)).
    \end{align*} Passing to a subsequence of the \(t_j\), we have that the image of \(w(t_j)\) is contained in \(B_{j^{-1}}(x_\infty)\). The same is true for \(t \geq t_j\) by the maximum principle argument in the second line of this proof. Since the coordinates on \(B'\) are normal, we thus have \(\|\Gamma(w(t))\|_\infty \to 0\) as \(t \to \infty\). The bound (\ref{regularballlemma:tensionboundsenergy}) thus follows in view of (\ref{regularballlemma:derivativebound}). \par By (\ref{regularballlemma:tensionboundsenergy}), we have for \(t\) large enough
    \begin{align*}
        \frac{d}{dt} \int |dw|^2 = -2\int |\tau(w)|^2 \leq -2C^{-1} \int |dw|^2.
    \end{align*} Hence, the energy of \(w\) is eventually exponentially decaying, so standard estimates for harmonic map flow imply that \(w(t) \to w_\infty\) exponentially in \(C^k\) for all \(k\), as desired.
\end{proof}

\begin{proof}[Proof of Theorem \ref{thm:regularballtheorem}]
    Let \(T > 0\) be the maximal time of smooth existence of \(w\). It follows from Lemma \ref{lemma:checkinghmfisMorreysolution} and Theorem \ref{thm:Morreysupbound} that if \(\delta_1 > 0\) is small enough, \(\|dw(t)\|_\infty\) is bounded at each time \(t \in (0, \min\{T, 1\})\). But since a supremum bound on \(dw\) is enough to ensure smooth existence of harmonic map flow, we deduce that \(T \geq 1\). Moreover, we have the bound \(\|dw(1)\|_\infty \leq C_{(\ref{Morreysupbound:criticalestimate})}\delta\), so if \(\delta_1\) is taken small enough, depending on the geometries of \(M\) and \(N\), \(dw(1)\) will be small in \(L^\infty\) such that the image of \(w(1)\) is be contained in a regular ball. The theorem now follows from Lemma \ref{lemma:regularballlemma}.
\end{proof}

To prove Cor. \ref{gapcor}, we will use an approximation result for maps whose derivatives have small Morrey norm. This follows the strategy of Struwe \cite[Prop. 7.2]{struwehigherdimhmf}, which originates from Schoen-Uhlenbeck \cite[Proposition on pg. 267]{schoenuhlenbeckhmbdryreg}
\begin{lemma}\label{Morreyapproximation}
    Let \(w \in W^{1, a}(M, N)\), where \(1 \leq a < \infty\). There exists \(\varepsilon > 0\), depending on the geometries of \(M\) and \(N\), such that the following holds. If for some \(\rho_0 > 0\) we have
    \begin{align}\label{Morreyapproximation:assumption}
        \sup_{x \in M, 0 < R \leq \rho_0} R^{2-n} \int_{B_R(x)} |dw|^2 \leq \varepsilon^2, 
    \end{align} then there exists \(\lambda_0 > 0\), depending on \(\rho_0\) and the geometries of \(M\) and \(N\), and a family of smooth maps \(w_\lambda, \lambda \in (0, \lambda_0]\), from \(M\) to \(N\) such that \(w_\lambda \to w\) in \(W^{1, a}(M, N)\) as \(\lambda \searrow 0\).  Moreover, if \(w\) is continuous, then \(w_\lambda\) is homotopic to \(w\) for each \(\lambda \in (0, \lambda_0]\). Finally, we have \(\|dw_\lambda\|_{2, 2} \leq C_{M, N}\|dw\|_{2, 2}\).
\end{lemma}
\begin{proof}
    By Nash's embedding theorem, we fix an isometric embedding of \(N\) into some \(\R^k\). Then, there is some \(\varepsilon_1 > 0\) such that the Euclidean \(\varepsilon_1\)-neighborhood \(U\) of \(N\) is a tubular neighborhood for which the nearest-neighbor projection \(\pi:B_{\varepsilon_1}(N) \to N\) is well defined and smooth. We first construct smooth approximations \(\tilde{w}_\lambda\) of \(w\) in \(W^{1, a}(M, \R^k)\) and show that they take value in \(U_{\varepsilon_1}\). \par
    Fix a finite cover \(\{B_{R_0}(p_i)\}\) of \(M\) by balls of radius \(R_0\), where \(R_0\) satisfies the conditions (1)-(4) from the proof of Prop. \ref{prop:Morreyheatestimates}. We also fix a smooth partition of unity \(\{\psi_i\}\) subordinate to the cover \(\{B_{R_0}(p_i)\}\), and for each \(i\) we fix \(p_i\)-centered geodesic coordinates
    \begin{align*}
        N_{p_i} : B_{6R_0}(0)\subset \R^n \to B_{6R_0}(p_i) \subset M.
    \end{align*} Let \(\eta\) be a standard mollifier on \(\R^n\), i.e., \(\eta \in C_c^\infty(B_1(0))\), \(\eta \geq 0\), \(\eta\) is radial, and \(\int_{\R^n} \eta = 1\). Denote \(\eta_\lambda := \lambda^{-n} \eta(\lambda^{-1} \cdot)\). For \(\lambda_0 \in (0, R_0)\), we may define the mollification for \(\lambda \in (0, \lambda_0]\)
    \begin{align*}
        \tilde{w}_{\lambda, i}(x) := \psi_i(x)\int_{B_{3R_0}(0)} w(N_{p_i}(N_{p_i}^{-1}(x)-y))\eta_\lambda(y) \, \mu_{\text{Euc}}(y).
    \end{align*}  Define \(\tilde{w}_\lambda := \sum_i \tilde{w}_{\lambda, i}\), which we note is smooth and furthermore is constant if \(w\) is constant. \par
    We first show that \(\tilde{w}_\lambda\) takes value in \(U_{\varepsilon_1}\). In the sequel, \(\lambda_0\) may decrease as need be, and \(C_{M, N}\), which is a positive constant depending on the geometries of \(M\) and \(N\), may increase as need be. By the uniform intersection property of the cover, the following holds. For each \(x \in M ,\) there exist \(J\) indices \(i_{x, 1}, \dots, i_{x, J}\), with \(J \leq C_n\), such that for all  \(R \in (0, R_0], \lambda \in (0, \lambda_0],\) and \(z \in B_R(x)\), we have \(\tilde{w}_\lambda(z) = \sum_{j = 1}^J \tilde{w}_{\lambda, i_{x, j}}(z)\) and \(\sum_{j = 1}^J \psi_{i_{x, j}}(z) = 1\). Given \(x \in M\), we compute
    \begin{align*}
        d_{\text{Euc}}(\tilde{w}_\lambda(x), N)^2 & \leq \text{Vol}(B_\lambda(x))^{-1} \int_{B_\lambda(x)} |\tilde{w}_\lambda(x) - w(y)|^2 \, \mu_g(y) \\
        &\leq C_{M, N}\lambda^{-n} \int_{B_\lambda(x)} |\sum_{j = 1}^J \tilde{w}_{\lambda, i_{x, j}}(x) - \psi_{i_{x, j}} w(y)|^2 \, \mu_g(y) \\
        &\leq C_{M, N}\lambda^{-n} \sum_{j = 1}^J\int_{B_\lambda(x)} | \tilde{w}_{\lambda, i_{x, j}}(x) - \psi_{i_{x, j}} w(y)|^2 \, \mu_g(y).
    \end{align*} 
     It follows from the Euclidean bounds in geodesic coordinates and the Poincar{\'e} inequality that
    \begin{align*}
        &\lambda^{-n} \int_{B_\lambda(x)} | \tilde{w}_{\lambda, i_{x, j}}(x) - \psi_{i_{x, j}} w(y)|^2 \, \mu_g(y) \leq C_{M, N}\lambda^{2-n} \int_{B_\lambda(x)} |\nabla (\psi_{i_{x, j}} w)|^2 \, \mu_g \\
        &\leq C_{M, N}\lambda^{2-n} \int_{B_\lambda(x)} |\nabla \psi_{i_{x, j}}|^2|w|^2 \, \mu_g + C_{M, N} \lambda^{2-n} \int_{B_\lambda(x)} |\psi_{i_{x, j}}|^2|\nabla w|^2 \, \mu_g \\
        &\leq C_{M, N} \lambda^2 + C_{M, N}\lambda^{2-n} \int_{B_\lambda(x)} |\nabla w|^2 \, \mu_g.
    \end{align*} If \(\lambda_0 \leq \rho_0\) and \(\varepsilon\) are small enough, this last line is less than \(\varepsilon_1^2\). Hence, up to shrinking \(\lambda_0\), we have
    \begin{align*}
        d_{\text{Euc}}(\tilde{W}_\lambda(x), N)^2  < \varepsilon_1^2.
    \end{align*} Thus, \(\tilde{w}_\lambda\) takes value in \(U\). \par 
    Consequently, we may define
    \begin{align*}
        w_\lambda := \pi \circ \tilde{w}_\lambda.
    \end{align*}By standard arguments, we have as \(\lambda \searrow 0\) that \(w_\lambda \to w\) in \(W^{1, a}(M, N)\), and in \(C^0(M, N)\) if \(w\) is continuous. Moreover, if \(w_\lambda \to w\) in \(C^0\), then \(w_\lambda\) is homotopic to \(w\) for \(\lambda\) small enough, and the same is true for all \(\lambda \in (0, \lambda_0]\) by continuity. \par 
    Finally, we prove the Morrey norm estimate on \(dw_\lambda\). 
    Let \(x \in M\) and \(R \in (0, R_0]\). Since \(|d\pi(v)| \leq C_{M, N}|v|\) for all tangent vectors \(v \in TU\), we have
    \begin{align*}
        R^{2-n}\int_{B_R(x)} |d w_\lambda|^2 \, \mu_g \leq C_{M, N}R^{2-n}\int_{B_R(x)} |\nabla  \tilde{w}_\lambda|^2 \, \mu_g.
    \end{align*} We next bound the RHS. Let \(z \in B_R(x)\). As before, we may write \(\tilde{w}_\lambda(z) = \sum_{j = 1}^J \tilde{w}_{\lambda, i_{x. j}}(z)\), and we have \(\sum_{j = 1}^J \psi_{i_{x, j}}(z) = 1\). Fixing \(x\)-centered geodesic coordinates \(N_x\), we have
    \begin{align*}
        \nabla \tilde{w}_\lambda(z) &= f_{\text{I}} + \sum_{j = 1}^J f_{\text{II}, j} + f_{\text{III}, j} + f_{\text{IV}, j},
    \end{align*} where
    \begin{align*}
        f_{\text{I}} &= \int_{B_{3R_0}(0)} \eta_\lambda(y) \nabla\big(w(N_x(N_x^{-1}(z) - y))\big) \, \mu_{\text{Euc}}(y) \\
        f_{\text{II}, j} &= \psi_{i_{x, j}}(z) \int_{B_{3R_0}(0)} \eta_\lambda(y) \nabla \left(w\left(N_{p_{i_{x, j}}}\left(N_{p_{i_{x, j}}}^{-1}(z)-y\right)\right)-w(N_x(N_x^{-1}(z) - y))\right) \, \mu_{\text{Euc}}(y) \\
        f_{\text{III}, j} &= (\nabla \psi_{i_{x, j}}(z))\int_{B_{3R_0}(0)} \eta_\lambda(y) \left(w\left(N_{p_{i_{x, j}}}\left(N_{p_{i_{x, j}}}^{-1}(z)-y\right)\right) - w(z)\right)\mu_{\text{Euc}}(y) \\
        f_{\text{IV}, j} &= (\nabla \psi_{i_{x, j}}(z))\int_{B_{3R_0}(0)} \eta_\lambda(y) \left(w(z)-w(\left(N_{x}\left(N_{x}^{-1}(z)-y\right)\right)\right)\mu_{\text{Euc}}(y).
    \end{align*} Since \(\int_{B_{3R_0(0)}} \eta_\lambda(y) \mu_{\text{Euc}}(y) = 1\), Jensen's inequality and the bounded geometry of \(M, N\) imply
    \begin{align*}
        R^{2-n} \int_{B_R(x)} |f_{\text{I}}(z)|^2 \, \mu_g(z) \leq C_{M, N} \|dw\|_{2, 2}^2.
    \end{align*}Similarly,
    \begin{align*}
        &C_{M, N}^{-1}\int_{B_R(x)} |f_{\text{II}}(z)|^2 \, \mu_g(z) \\
        &\leq \int_{B_R(x)} \left|\int_{B_{3R}(0)} \eta_\lambda(y) \nabla\left(w\left(N_{p_{i_{x, j}}}\left(N_{p_{i_{x, j}}}^{-1}(z)-y\right)\right)\right) \, \mu_{\text{Euc}}(y) \right|^2 \, \mu_g(z) +\int_{B_R(x)} |f_{\text{I}}(z)|^2 \, \mu_g(z) \\
        &\leq C_{M, N} R^{n-2} \|dw\|_{2, 2}^2.
    \end{align*} Now, since
    \begin{align*}
        w(z) - w\left(N_x\left(N_x^{-1}(z)-y\right)\right) = \int_0^1 \frac{d}{dt} w\left(N_x\left(N_x^{-1}(z) - (1-t)y\right)\right) \, dt,
    \end{align*} we have 
    \begin{align*}
        R^{2-n}\int_{B_R(x)} |f_{\text{IV}, j}(z)|^2 \, \mu_g(z) \leq C_{M, N} \|dw\|_{2, 2}^2.
    \end{align*} Similarly,
    \begin{align*}
        R^{2-n}\int_{B_R(x)} |f_{\text{III}, j}(z)|^2 \, \mu_g(z) \leq C_{M, N} \|dw\|_{2, 2}^2.
    \end{align*}
    Thus, we deduce that
    \begin{align*}
        R^{2-n} \int_{B_R(x)} |\nabla \tilde{w}_\lambda|^2 \, \mu_g \leq C_{M, N} \|d w\|_{2, 2}^2,
    \end{align*} from which the desired bound follows.
\end{proof}
\begin{proof}[Proof of Corollary \ref{gapcor}]

    The ``only if'' direction is clear since the infimum is attained by a constant map.
    
    On the other hand, suppose that (\ref{Morreyinfimum}) holds. Then, there exist continuous maps \(w_i \in \gamma\) such that \(\|dw_i\|_{2, 2} \searrow 0\). Then by Lemma \ref{Morreyapproximation}, there exist smooth maps \(w_i'\) such that \(w_i'\) is homotopic to \(w_i\) and \(\|dw_i'\|_{2, 2} \searrow 0\). Hence for \(i\) large enough, Theorem \ref{thm:regularballtheorem} implies that \(w_i'\) is homotopic to a constant map. We conclude that \(\gamma\) is trivial, as desired.
\end{proof}

\subsection*{Acknowledgments}
The author thanks his advisor Alex Waldron for guidance and simplifications. The author was supported by NSF DMS-2037851 during the preparation of this article.
\subsection*{Data Availability Statement}
The author declares that data sharing is not applicable to this article as no datasets were generated or analyzed during the current study.
\section{Appendix}
\begin{proof}[Proof of Lemma \ref{lemma:gradientestimates}]
     In the sequel, \(C\) denotes a constant depending on the geometry of \(M\) that may increase in each appearance. For the first estimate, it suffices that we have for \(t \in (0, 1]\) the estimate
    \begin{align}\label{chengliyaubound}
        |\nabla_x H(x, y, t)| \leq Ct^{-\frac{1}{2}} H\left(x, y, Ct\right).
    \end{align} To see that this bound is sufficient, note that
    \begin{align*}
        \left|\nabla_x \int H(x, y, t) f_0(y) \, \mu_g(y)\right| &\leq \int \left|\nabla_x H(x, y, t)\right| \left|f_0(y)\right| \, \mu_g(y) \\
        &\leq Ct^{-\frac{1}{2}}\int H(x, y, Ct) |f_0(y)| \, \mu_g(y),
    \end{align*} so by \(L^p\) estimates for the heat kernel on \((0, C]\),
    \begin{align*}
        \|\nabla H_t(f_0)\|_r \leq Ct^{-\frac{1}{2}} \|H_{Ct}(|f_0|)\|_r \leq Ct^{-\frac{n}{2}\left(\frac{1}{q}-\frac{1}{r}\right)-\frac{1}{2}}\|f_0\|_q.
    \end{align*} But the bound (\ref{chengliyaubound}) follows from \cite[Theorem 6]{chengliyauheatkernel} and (\ref{heatkernel:lowerbound}), so the first estimate of Lemma \ref{lemma:gradientestimates} holds.
   \par 
   For the second estimate of Lemma (\ref{lemma:gradientestimates}), we do the following. First we prove the estimate with \(q = r\), and then we prove it with \(r = \infty\) and \(q = 1\). Then, the other cases follow by interpolating as usual. We start with the \(q = r\) case. \par Recall that if \(f\) solves the heat equation, then 
    \begin{align}\label{gradientestimates:oneformheatequation}
        (\partial_t - \Delta_1)\nabla f = \text{Ric} \# \nabla f,
    \end{align} where \(\Delta_1\) is the connection Laplacian on one-forms. Since \(M\) is compact, we let \(H_t^1\) denote the fundamental solution of the heat equation on one-forms. Then by (\ref{gradientestimates:oneformheatequation}) and Duhamel's principle,
   \begin{align*}
       \nabla H_t(f_0) = H_t^1(\nabla f_0) + \int_0^t H_{t-s}^1 \left(\text{Ric} \# \nabla H_s( f_0)\right) \, ds.
   \end{align*} Since for all \(r\) and \(t \in (0, 1]\)
   \begin{align*}
       \left\|H_t^1\right\|_{L^r \to L^r} \leq C,
   \end{align*} we have
   \begin{align*}
       \left\|\nabla H_t(f_0)\right\|_r &\leq \|H_t^1(\nabla f_0)\|_r + \int_0^t \left\|H_{t-s}^1\left(\text{Ric} \# \nabla H_s( f_0)\right)\right\|_r \, ds \\
       &\leq C\|\nabla f_0\|_r + C\int_0^t \left\|\text{Ric} \# \nabla H_s( f_0)\right\|_r \, ds \\
       &\leq C\|\nabla f_0\|_r + C\int_0^t \left\|\nabla H_s(f_0)\right\|_r \, ds.
   \end{align*} Hence, if \(t\) is sufficiently small depending only on \(C\),
   \begin{align*}
       \sup_{0 < s \leq t} \|\nabla H_s(f_0)\|_r \leq C\|\nabla f_0\|_r + \frac{1}{2}\sup_{0 < s \leq t} \|\nabla H_s(f_0)\|_r,
   \end{align*} whence
   \begin{align}\label{gradientestimates:LrLrbound}
       \|\nabla H_t(f_0)\|_r \leq C\|\nabla f_0\|_r.
   \end{align} By the semigroup property of \(H_t\), we obtain this estimate with a larger \(C\) for all \(t \in (0, 1]\). This concludes the \(q = r\) case of the second estimate of Lemma \ref{lemma:gradientestimates}. \par 
   Next we do the \(r = \infty\) and \(q = 1\) case of the second estimate of Lemma \ref{lemma:gradientestimates}. By (\ref{gradientestimates:oneformheatequation}) and Kato's inequality, we have (weakly) 
    \begin{align}\label{gradientestimates:oneformheatinequality}
        (\partial_t - \Delta)|\nabla f| \leq C|\nabla f|.
    \end{align} Thus by parabolic Moser iteration, we have for \(t \in (0, 1]\)
   \begin{align*}
       \|\nabla H_t(f_0)\|_\infty \leq Ct^{-\frac{n+2}{2}}\int_0^t \|\nabla H_s(f_0)\|_1 \, ds.
   \end{align*} Then by (\ref{gradientestimates:LrLrbound}) with \(r = 1\)
   \begin{align*}
       \|\nabla H_t(f_0)\|_\infty \leq Ct^{-\frac{n}{2}}\|\nabla f_0\|_1,
   \end{align*} which is the \(r = \infty\) and \(q = 1\) case. \par 
   The other cases follow by interpolating the \(L^r \to L^r\) bound and the \(L^1 \to L^\infty\) bound.
\end{proof}

\bibliographystyle{plain}
\bibliography{biblio}
\end{document}